\documentclass[11pt]{article}

\usepackage{amsmath,amssymb,amsfonts}
\usepackage{bm}
\usepackage{amsthm}
\usepackage{graphicx}
\usepackage{booktabs}
\usepackage{float}
\usepackage{xcolor}
\usepackage{xurl}
\usepackage[left=0.7in,right=0.7in,top=0.8in,bottom=0.8in]{geometry}
\usepackage{tikz}
\usetikzlibrary{arrows.meta,positioning,calc}

\usepackage{subcaption}

\usepackage{algorithm}
\usepackage{algorithmic}

\usepackage[hidelinks]{hyperref}

\graphicspath{{pic/}}

\DeclareGraphicsExtensions{.pdf,.png,.jpg,.jpeg}

\theoremstyle{plain}
\newtheorem{theorem}{Theorem}

\newtheorem{proposition}{Proposition}

\theoremstyle{definition}
\newtheorem{definition}{Definition}

\theoremstyle{remark}

\begin{document}
	
\title{Algebraic Invariant Quadratization Schemes for Cahn--Hilliard Equations}
\author{Fei Xie$^{1}$ \quad Nan Lu$^{2}$ \quad
Yajuan Sun$^{1}$\thanks{The work of the third author is supported by the National Natural Science Foundation of China (Grant No. 12271513).}\\[0.75em]
\small $^{1}$State Key Laboratory of Mathematical Sciences, Academy of Mathematics and Systems Science,\\
\small Chinese Academy of Sciences, Beijing 100190, China; University of Chinese Academy of Sciences, Beijing 100049, China\\
\small $^{2}$Department of Mathematics, Shanghai Normal University, Shanghai 200234, China\\[0.25em]
\small \texttt{xiefei2021@lsec.cc.ac.cn} \quad
\texttt{lunan@shnu.edu.cn} \quad
\texttt{sunyj@lsec.cc.ac.cn}}
\date{}
\maketitle
	
\begin{abstract}
In this paper, we propose the Algebraic Invariant Quadratization (AIQ) framework for rational-like energy functions by introducing auxiliary variables, which are interpreted as Casimir functions of the extended system. Combining AIQ with symplectic Runge--Kutta (SRK) methods in time and Fourier pseudo-spectral  discretization in space, we obtain fully discrete schemes. The resulting schemes are applied to Cahn--Hilliard equations in both the isotropic and anisotropic cases. We analyze the discrete dispersion relation, spinodal instability, coarsening behavior, and missing-orientation phenomena. Numerical comparisons demonstrate the improved performance superiority of the proposed method over the stabilized invariant energy quadratization (S-IEQ) and scalar auxiliary variable (SAV) methods in preserving the original energy evolution and capturing the underlying physical phenomena.

\end{abstract}


	
\section{Introduction}
For many physical systems, preserving the intrinsic geometric structures is essential for accurately capturing their qualitative behavior. These structures may include symplecticity, mass conservation, momentum conservation, Casimir invariants, and energy conservation or dissipation. Typical examples arise from Hamiltonian systems in celestial mechanics, plasma and fluid models with geometric invariants, and biological or phase-field models governed by dissipative dynamics. In systems where the dynamics are constrained by energy surfaces or driven by energy dissipation, the energy law plays an important role in stability and long-time evolution. A numerical method that fails to respect this structure may introduce artificial energy drift and lead to unreliable long-time predictions. Therefore, constructing energy-preserving or energy-dissipative schemes is a central issue in structure-preserving numerical methods.

Various numerical methods have been developed to preserve energy structures, including the discrete variational derivative method~\cite{Furihata2010DVD}, the average vector field (AVF) method~\cite{Celledoni2012}, and discrete-gradient-based integral-preserving frameworks for PDEs~\cite{DahlbyOwren2011}. These methods usually require more complicated constructions and may involve numerical quadrature. Recently, auxiliary-variable methods such as IEQ~\cite{Yang2017IEQ} and SAV~\cite{Shen2018SAV} have provided more systematic and efficient reformulation strategies and have been widely used for gradient flows. These methods usually require less detailed structural analysis and are easier to construct and implement. They are effective and flexible. However, they usually preserve a modified energy rather than the original energy itself, which may lead to energy errors or drift and affect the accuracy of long-time physical simulations.

To maintain consistency with the original energy law and improve the reliability of long-time simulations, recent studies have focused on auxiliary-variable formulations that preserve the original energy structure. For Hamiltonian systems with polynomial first integrals, the multiple quadratic auxiliary variable (MQAV) method was introduced in \cite{Tapley2022} to preserve the original invariants. More recently, for systems reformulated in linear-gradient form, we establish a dimension-raising framework \cite{Lu2025}. In this framework, the original system is embedded into an extended space through quadraticized auxiliary variables which can be identified as Casimir functions of the extended dynamics. Therefore, for the extended system.

Preserving the modified quadratic energy alone is not sufficient to recover the original energy law. The Casimir constraints must also be preserved. Indeed, the preservation of these constraints ensures that the numerical solution remains on the constraint manifold, where the extended formulation is equivalent to the original one. This observation also explains why standard IEQ and SAV methods do not automatically guarantee consistency with the original energy.

Motivated by this idea, we further study systems with rational-like invariants and establish an Algebraic Invariant Quadratization (AIQ) framework. The key point is to construct an extended formulation with two quadratic structures: a quadratic energy for the extended system and quadratic auxiliary constraints. We	prove that rational-like invariants can be quadratized by introducing quadratic auxiliary variables, which provides the theoretical foundation for the AIQ construction. We combine the AIQ framework with SRK time discretizations and Fourier pseudo-spectral spatial discretizations. The resulting fully discrete schemes preserve both the quadratic energy of the extended system and the Casimir constraints. Thereby recovering the original energy law on the constraint manifold.

We use Cahn-Hillirad (CH) equations as model problems to validate the proposed AIQ framework, The CH equations are fundamental phase-field models for phase separation and interface-driven pattern formation in materials science~\cite{Boettinger2002,cahn1958}. They can be formulated as an \(H^{-1}\) gradient flows associated with a Ginzburg--Landau free energy and have two basic structural properties: mass conservation and energy dissipation~\cite{Wu2022Review}. These features make CH-type equations a natural class of models for applying and testing original-energy-consistent structure-preserving methods. In many applications, the surface energy gives rise to direction-dependent surface tension and equilibrium shapes characterized by the Wulff construction~\cite{CahnHoffman1974,Wulff1901}. 

In such models, the anisotropic factors often lead to more complicated energy densities involving polynomial, rational, and algebraic branch functions, which makes anisotropic CH equations particularly suitable for the AIQ framework. In this paper, we apply AIQ to both isotropic and anisotropic CH models. The rational-like energy terms can be rewritten as quadratic functions in an extended space, and the original energy law is recovered through the preservation of auxiliary constraints. We prove the corresponding original discrete energy-dissipation laws at the fully discrete level. Numerical experiments verify the temporal accuracy, dispersion behavior, coarsening dynamics, and missing-orientation phenomena, and show improved original-energy consistency compared with S-IEQ and SAV.

The outline of this paper is as follows. Section~2 presents the AIQ	framework for quadratizing rational-like invariants and verifies the structure-preserving property of the extended system. Section~3 develops	fully discrete schemes for isotropic and anisotropic CH equations and proves their discrete energy-dissipation laws and discrete mass conservation. Section~4 examines the discrete dispersion relation and related physical phenomena, including spinodal instability, coarsening behavior, and missing orientations. Section~5 verifies the temporal accuracy, compares the original-energy behavior with existing auxiliary-variable methods, and investigates the evolution dynamics from different initial conditions. Finally, Section~6 concludes the paper with a brief summary and several directions for future work.	
	
\section{Quadratization-based numerical methods}
We begin this section by introducing the definition of rational-like functions.
\begin{definition}[Rational-like Function]\label{defn:rational}
	A function \(H(x)\) defined on \(\mathbb R^n\) is called a rational-like function if there exists a nonzero polynomial of degree $d$ in $n+1$ variables $P(x,y)=\sum_{i=0}^d p_i(x)\,y^i$, such that for all $x\in\mathbb{R}^n$,
	\[
	P\bigl(x,\;H(x)\bigr)
	\equiv 0.
	\]
\end{definition}
Rational-like functions form a subclass of algebraic functions characterized by the above structure, which facilitates energy reformulation.

\begin{definition}[Dimension-raising Transformation]
	Suppose \(G(x,y):\mathbb{R}^{n+d} \to \mathbb{R}^d\) is a continuously differentiable function, and assume that
	\(\partial_yG(x,y)\in\mathbb{R}^{d\times d}\) is nonsingular whenever \(G(x,y)=0\).
	For a given \(H:\mathbb{R}^n \to \mathbb{R}\), \(G\) is called a dimension-raising transformation if \(H\) can be embedded into a higher-dimensional function
	\(\widetilde H:\mathbb{R}^{n+d} \to \mathbb{R}\) such that
	\begin{equation}\label{eq:restrict}
		\widetilde H(x,y)=H(x),\qquad
		(x,y)\in\mathcal M:=\{(x,y)\in\mathbb{R}^{n+d}:G(x,y)=0\}.
	\end{equation}
\end{definition}

A function \(H:\mathbb{R}^n\to\mathbb{R}\) is called quadraticizable if there exist an integer \(d\ge1\), a constraint map \(G=(G_1,\ldots,G_d)^\top:\mathbb{R}^{n+d}\to\mathbb{R}^d\), and a quadratic function \(\widetilde H:\mathbb{R}^{n+d}\to\mathbb{R}\) satisfying  \eqref{eq:restrict}, where each component \(G_i\) is a quadratic function of \((x,y)\). In other words, \(H\) can be represented as the restriction of a quadratic function in a higher-dimensional space under quadratic constraints. This construction is called quadratic dimension raising. We next show that every rational-like function is quadraticizable.
\begin{theorem}\label{thm:quar}
	Assume that \(H:\mathbb R^n\to\mathbb R\) is a rational-like function. Then \(H\) is quadraticizable.
\end{theorem}
\begin{proof}
	Since the rational-like functions are generated by finitely many algebraic operations, we verify the claim for the four elementary cases.
	
	Firstly, assume that $H(x)$ is a monomial of $n$-variables with even degree $d=2k$. By relabeling each occurrence of a variable as a distinct symbol, we write every repeated factor $x_i$ as $x_{\alpha_j}$ with a new index $\alpha_j$.
	Then $H(x)$ can be expressed by
	\[
	H(x)=x_{\alpha_1}x_{\alpha_2}\cdots x_{\alpha_{2k-1}}x_{\alpha_{2k}},
	\]
	where $\alpha_{i}, i=1,\cdots, 2k$ take values from 1 to $2k$. Let $y_{i}=x_{\alpha_{2i-1}}\,x_{\alpha_{2i}}$, we can define a dimension-raising transformation whose \(i\)-th component is \(G_i(x,y)=y_{i}-x_{\alpha_{2i-1}}\,x_{\alpha_{2i}}\), which is quadratic. Then \(\tilde{H}(x,y)=y_1 \cdots y_k\) agrees with \(H(x)\) on the constraint manifold and has degree \(k\) in the auxiliary variables. If \(k>1\), we apply the same procedure to the product \(\prod_{i=1}^k y_i\).
	
	Secondly, considers a monomial of odd degree $d=2k+1$, namely,
	\[
	H(x)=x_{\alpha_1}\cdots x_{\alpha_{2k}}\,x_{\alpha_{2k+1}}
	=\Bigl(\prod_{j=1}^{2k}x_{\alpha_j}\Bigr)\,x_{\alpha_{2k+1}}.
	\]
	The first term is a monomial of  degree $2k$. Thus, using the same auxiliary variable as in the first case, we can define  $\tilde{H}(x,y)=H(x)=\Bigl(\prod_{i=1}^ky_i\Bigr)\,x_{\alpha_{2k+1}}$.
	After finitely many steps, the original monomial is represented by a quadratic extended function subject to quadratic constraints, and hence $H$ can be quadraticized.
	
	In the third case, suppose \(H(x)\) is a function with a rational power. If \(H(x)=1/x\), introducing an auxiliary variable \(y=1/x\) gives \(\tilde{H}(x, y)=y\). The corresponding dimension-raising quadratic function is \(G(x,y) = xy - 1\). 
	
	%
	In the last case, consider a rational power $H(x)=x^{1/q}$ with $q\in\mathbb{N}$, $q\ge2$. From the binary expansion we have $q=\sum_{i=0}^{s} b_i 2^i$. We introduce auxiliary variable $u_j$, $v_j$, $i=0,1,\cdots s$ intend to satisfy $u_0^q=x$. 
	Let \(q=\sum_{j=0}^{s} b_j2^j\) be the binary representation of \(q\), with	\(b_j\in\{0,1\}\). We introduce auxiliary variables \(u_i\) and \(v_i\) by \(u_0=x^{1/q}\), \(v_0=1\), and
	\[
	u_{i+1}=u_i^2,\qquad v_{i+1}=v_i u_i^{b_i},\qquad i=0,\ldots,s-1.
	\]
	Then \(x=v_su_s^{b_s}\), and the construction involves only quadratic relations	between the auxiliary variables.
	Once $u_0=x^{1/q}$ is enforced by the quadratic constraints, the general case $x^{p/q}=(u_0)^p$ follows, and the monomial construction applies to $(u_0)^p$. This concludes the proof.
\end{proof}
For example, we consider the fourfold anisotropy function
\[
\Gamma(x_1,x_2)
=
1-3\alpha+4\alpha
\frac{x_1^4+x_2^4}{(x_1^2+x_2^2)^2}.
\]
According to Definition~\ref{defn:rational}, \(\Gamma\) is a rational-like function. By Theorem~\ref{thm:quar}, it is quadraticizable. More explicitly, introducing the auxiliary variables
\begin{equation}\label{eg:fourfold}
	y_1=x_1^2,\quad
	y_2=x_2^2,\quad
	y_3=(y_1+y_2)^2,\quad
	y_4=y_3^{-1},\quad
	y_5=y_1^2+y_2^2,\quad
	y_6=y_4y_5,
\end{equation}
we obtain \(\widetilde\Gamma(y_1,y_2,y_3,y_4,y_5,y_6)=1-3\alpha+4\alpha y_6\), which is a linear function of the auxiliary variable \(y_6\) on the constraint set.

Consider an ODE system of the form
\begin{align}\label{eq:ODE}
\dot x= f(x).
\end{align}
If the system has a conservative or dissipative structure, then it can be rewritten as
\begin{align}\label{eq:gradientODE}
\dot x=\mathbb{M}(x)\nabla H(x)
\end{align}
where \(\mathbb M(x)\) is skew-symmetric or negative semidefinite\cite{McLachlanQuispelRobidoux1999}.


For the system~\eqref{eq:ODE}, Theorem~\ref{thm:quar} shows that if the quantity \(H(x)\) is rational-like, then \(H\) can be represented by a quadratic function \(\widetilde H(x,y)\) in an extended space, together with a constraint map \(G(x,y)\) whose components are quadratic. However, the dimension-raising transformation \(G(x,y)\) is not unique. If the \(y\)-Jacobian \(\partial_yG(x,y)\) is nonsingular on the constraint manifold, then the implicit function theorem implies that there exists a smooth function \(\varphi\) such that 
\[ G\bigl(x,\varphi(x)\bigr)\equiv 0. \] 
Thus the auxiliary variables can be written as \(y=\varphi(x)\) on the constraint manifold.

Introduce the homeomorphic  map $j:\mathbb{R}^n\to \mathcal{M}, j(x)=(x,\varphi(x))$. It is clear that
\[
\tilde H(x,y)\big|_\mathcal{M}
= \tilde H\bigl(j(x)\bigr)
= H(x).
\]
In what follows, we demonstrate that this extended system not only exactly preserves the original invariants but also has a variety of geometric and algebraic properties.

Suppose each component \(G_i(x,y)\) of the dimension-raising transformation is a homogeneous quadratic polynomial. Let \(z=(x^\top,y^\top)^\top\in\mathbb{R}^{n+d}\). Then \(G_i(x,y)=\frac12\,z^\top Q^i z,\; i=1,\ldots,d,\) where \(Q^i\in\mathbb{R}^{(n+d)\times(n+d)}\) is symmetric and can be partitioned as
\[
Q^i=\begin{pmatrix}
Q_{11}^i & Q_{12}^i\\
(Q_{12}^i)^\top & Q_{22}^i
\end{pmatrix}.
\]
The Jacobian $DG(x,y)\in\mathbb{R}^{d\times(n+d)}$ is given by
\[
\begin{aligned}
D G&=
\begin{pmatrix}
	\nabla G_1^\top \\
	\vdots  \\[0.75ex]
	\nabla G_d^\top\\
\end{pmatrix}
&=\left(\underbrace{
	\begin{array}{c}
		x^\top Q_{11}^1 + y^\top Q_{12}^{1\,\top} \\[0.75ex]
		\vdots \\[0.75ex]
		x^\top Q_{11}^d + y^\top Q_{12}^{d\,\top}
	\end{array}
}_{A(x,y)}
\;\Bigg|\;
\underbrace{
	\begin{array}{c}
		x^\top Q_{12}^1 + y^\top Q_{22}^1 \\[0.75ex]
		\vdots \\[0.75ex]
		x^\top Q_{12}^d + y^\top Q_{22}^d
	\end{array}
}_{B(x,y)}\!\right).
\end{aligned}
\]
For simplicity of the matrix representation, we assume that \(B(x,y)\) has constant full rank on the considered domain. In a more general implicit formulation, this assumption is not necessary. Under the assumption, the Moore--Penrose inverse \(B^\dagger(x,y)=
\bigl(B(x,y)^\top B(x,y)\bigr)^{-1}B(x,y)^\top\) is well defined and smooth along the relevant trajectories.

For a given system~\eqref{eq:ODE}, the corresponding extended system is
\begin{equation}\label{eq:extendODE}
\begin{pmatrix}
	\dot{x}\\
	\dot{y}
\end{pmatrix}
=F(x,y):=
\begin{pmatrix}
	\tilde f(x,y)\\[0.2em]
	-\,B^\dagger(x,y)\,A(x,y)\,\tilde f(x,y)
\end{pmatrix},
\end{equation}

Assume that $\mathcal{M}$ can be represented as a graph $y=\varphi(x)$. Then the restriction of $F$ to $\mathcal{M}$ is $F(x,y)\big|_{\mathcal{M}}=F(x,\varphi(x))=F(j(x))$. Since the $x$-component of the extended system~\eqref{eq:extendODE} is $\dot x=\tilde f(x,y)$, consistency with the original system $\dot x=f(x)$ on $\mathcal{M}$ implies
\begin{equation}\label{eq:dotx}
\tilde f\bigl(j(x)\bigr)=f(x).
\end{equation}
Moreover, along trajectories on $\mathcal{M}$ we have $y(t)=\varphi(x(t))$, and the chain rule gives $\dot y(t)=\varphi'\bigl(x(t)\bigr)\,\dot x(t)=\varphi'\bigl(x(t)\bigr)\,f\bigl(x(t)\bigr)$. On the other hand, the $y$-component of system~\eqref{eq:extendODE} is $\dot y=-B^\dagger(x,y)\,N(x,y)\,\tilde f(x,y)$, and therefore, on $\mathcal{M}$,
\begin{equation}\label{eq:doty}
-\,B^\dagger\bigl(j(x)\bigr)\,N\bigl(j(x)\bigr)\,f(x)
=\varphi'(x)\,f(x).
\end{equation}
Combining Eq.~\eqref{eq:dotx} and Eq.~\eqref{eq:doty} yields
\begin{equation}\label{eq:F_restriction}
F(j(x))=D(j(x))\,f(x)
=
\begin{pmatrix}
	I\\
	\varphi'(x)
\end{pmatrix}f(x).
\end{equation}
Here $D(j(x))$ denotes the Jacobian matrix of $j(x)$.

A scalar function $C(x)$ is called a Casimir invariant associated with
the structure matrix $\mathbb{M}(x)$ if its gradient lies pointwise in the null space of $\mathbb{M}(x)$, i.e., \(\mathbb{M}(x)\nabla C(x)=0\). Thus, a Casimir invariant is determined by the null space of the structure matrix, and is therefore preserved along the corresponding flow.

According to our result in \cite{Lu2025}, the constraint components \(G_i(x,y)\) can be interpreted as Casimir invariants of the extended system. Hence, to recover the original energy law, one must preserve not only the extended energy \(\widetilde H(x,y)\), but also the Casimir constraints \(G(x,y)=0\). This keeps the numerical solution on the constraint manifold.

\begin{proposition}
	The extended system~\eqref{eq:extendODE} is equivalent to the original system. Moreover, it inherits the same conservative (or dissipative) property as the original system.
\end{proposition}

\begin{proof}
	Let \(F(x,y)\) denote the vector field of the extended system. On the constraint manifold \(\mathcal{M}\), \(\tilde H(x,y)\) coincides with the original energy \(H(x)\), and the extended dynamics reduces to the original one in the \(x\)-component. Hence, by construction,
	\begin{equation}\label{eq:verify_pre}
		\nabla \tilde H(x,y)^\top F(x,y)\big|_{\mathcal{M}}
		=
		\nabla H(x)^\top f(x).
	\end{equation}
	Therefore,
	\[
	\nabla \tilde H(x,y)^\top F(x,y)\big|_{\mathcal{M}}
	=
	\begin{cases}
		0, & \text{if the original system is conservative},\\[0.4em]
		\le 0, & \text{if the original system is dissipative}.
	\end{cases}
	\]
	Hence, the extended system inherits the same conservation or dissipation property on \(\mathcal{M}\).
	
	The equivalence between the extended system and the original system on the constraint manifold follows from~\cite{Lu2025}. Also, \(G\) is conserved along the extended flow, i.e.,
	\[
	\frac{d}{dt}G(x(t),y(t))=0.
	\]
\end{proof}

An $s$-stage symplectic Runge--Kutta (SRK) method~\cite{Hairer2006} has coefficients $(a_{ij},b_i)$ satisfying
\begin{align}\label{eq:srk-symplectic-condition}
b_i a_{ij}+b_j a_{ji}-b_i b_j=0,
\qquad 1\le i,j\le s .
\end{align}
This condition is essential for preserving quadratic invariants under Runge--Kutta discretizations. At the discrete level, the key point is that SRK methods preserve quadratic invariants. We prove below that this property applies to both the extended quadratic energy and the quadratic Casimir constraints in the extended system.
\begin{proposition}
	Applying an SRK method to the extended system~\eqref{eq:extendODE} 	yields a structure-preserving discretization that exactly preserves all quadratic invariants. If the initial data satisfy the constraint $G(x_0,y_0)=0$, the numerical solution remains on the constraint manifold for all time steps. Eliminating the auxiliary variables via the constraint then produces a method that preserves the original invariants of the original system. Moreover, if the SRK method is of order $p$, so is the reduced integrator.
\end{proposition}

\begin{proof}
	SRK methods preserve quadratic invariants; hence
	\[
	G(x_{n+1},y_{n+1})=G(x_n,y_n),\quad \widetilde H(x_{n+1},y_{n+1})=\widetilde H(x_n,y_n).
	\]
	Due to the implicit function theorem, it is easy to get \(y=\varphi(x)\) at the grid point from \(G(x,y)=0\). Substituting this relation into the SRK discretization of the extended system yields a reduced integrator for the original system:
	\[
	\begin{aligned}
		&X_i=x_n+h\sum_{j=1}^s a_{ij}\widetilde f(X_j,\varphi(X_j)),
		\qquad i=1,\dots,s,	\\
		&x_{n+1}=x_n+h\sum_{i=1}^s b_i\widetilde f(X_i,\varphi(X_i)).
	\end{aligned}
	\]
	The numerical algorithm has	the same order as the underlying SRK method.
\end{proof}
The above dimension-raising idea can also be extended to infinite-dimensional systems, where the energy is described by a functional. Let $\Omega\subset\mathbb R^d$ be a periodic domain. We consider $u$ in a sufficiently smooth Sobolev space so that all quantities below are well defined.

The energy functional is given by
\begin{equation}\label{eq:H_pde}
\mathcal{H}[u]=\int_{\Omega} H\bigl(\operatorname{Pr}^{(m)}u\bigr)\,dx,
\end{equation}
where $\operatorname{Pr}^{(m)}u$ denotes $u$ and its derivatives up to order $m$. We consider evolution equations of the form
\begin{equation}\label{eq:ori_pde}
u_t=\mathcal D\frac{\delta\mathcal H}{\delta u},
\end{equation}
which satisfy the formal energy identity \(\frac{d}{dt}\mathcal H=\left(\mathcal D\frac{\delta\mathcal H}{\delta u},\frac{\delta\mathcal H}{\delta u}\right)\), where $\mathcal D$ is skew-adjoint in the conservative case and negative semidefinite in the dissipative case.

\begin{definition}
	The energy functional $\mathcal H[u]$ is called quadratic if its density is a quadratic form in $\operatorname{Pr}^{(m)}u$, i.e.,
	\[
	H(\operatorname{Pr}^{(m)}u)=\frac12\bigl(\operatorname{Pr}^{(m)}u\bigr)^\top A\,\operatorname{Pr}^{(m)}u,
	\]
	where $A$ is a symmetric constant matrix.
\end{definition}
For such functionals the SRK method inherits a precise structure--preserving property.
\begin{theorem}
	Assume that \(\mathcal H[u]\) is quadratic and there exists a symmetric linear operator \(S\) such that
	\[
	\frac{\delta \mathcal H}{\delta u}=Su.
	\]
	Let an SRK method with coefficients \((a_{ij},b_i)\) be	applied to system~\eqref{eq:ori_pde}. Then the fully discrete solution satisfies
	\[
	\mathcal H[u_{n+1}]-\mathcal H[u_n]
	=
	\Delta t\sum_{i=1}^s b_i
	\bigl(\mathcal D S U_i,\,S U_i\bigr) \begin{cases} =0, & \text{if } \mathcal D \text{ is skew-adjoint},\\[0.4em] \le 0, & \text{if } \mathcal D \text{ is negative semidefinite}, \end{cases}
	\]
	where \(U_i\) are the SRK stage values.
\end{theorem}
\begin{proof}
	Let \(K_i=\mathcal D S U_i\). Then
	\[
	u_{n+1} = u_n + \Delta t \sum_{i=1}^s b_i K_i, \qquad
	U_i = u_n + \Delta t \sum_{j=1}^s a_{ij} K_j.
	\]	
	For the quadratic functional \(\mathcal H[u] = \frac12(u,Su)\),
	\[
	\mathcal H[u_{n+1}] - \mathcal H[u_n]
	= \Delta t \sum_{i=1}^s b_i (K_i, S U_i)
	= \Delta t \sum_{i=1}^s b_i (\mathcal D S U_i, S U_i),
	\]
	where the last equality uses the SRK symplectic condition \eqref{eq:srk-symplectic-condition}.
	
	If $\mathcal D$ is skew-adjoint, $(\mathcal D S U_i, S U_i)=0$. If $\mathcal D$ is	negative semidefinite and $b_i \ge 0$, then $(\mathcal D S U_i, S U_i) \le 0$ for each	stage $i$. This completes the proof.
\end{proof}
\begin{theorem} 
	Assume that system~\eqref{eq:ori_pde} admits a rational-like energy density. Then there exist auxiliary variables \(\mathbf y(x)\in\mathbb R^r\) and dimension-raising constraints 
	\[ 
	G_i\bigl(\operatorname{Pr}^{(n)}u,\mathbf y\bigr)=0, \qquad i=1,\ldots,r, 
	\] 
	where \(G_i\) is quadratic with respect to \((\operatorname{Pr}^{(n)}u,\mathbf y)\), and the system can be embedded into an extended space \((\operatorname{Pr}^{(n)}u,\mathbf y)\) with a quadratic energy functional 
	\[ 
	\widetilde{\mathcal H}[u,\mathbf y] = \int_{\Omega} \widetilde H\bigl(\operatorname{Pr}^{(n)}u,\mathbf y\bigr)\,dx, 
	\] 
	where \(\widetilde H\) is quadratic. Moreover, the equivalence \( \widetilde{\mathcal H}[u,\mathbf y]=\mathcal H[u] \) holds on the constraint manifold defined by \(G_i=0\), \(i=1,\ldots,r\). \end{theorem}

\begin{proof}
	Denote $u_i$ be the \(i\)-th component of \(\operatorname{Pr}^{(n)}u\), then $\operatorname{Pr}^{(n)}u=(u_0,u_1,\cdots, u_n)$ with $u_0=u$. The rational-like density $H(\operatorname{Pr}^{(n)}u)$ has the same structure as in Theorem~\ref{thm:quar}. Therefore, we apply the same quadratization procedure to obtain a quadratic extended representation.
\end{proof}

This section has developed the Algebraic Invariant Quadratization (AIQ) framework. The main idea is to rewrite nonlinear rational-like energies as quadratic energies in an extended variable space by introducing suitable auxiliary variables whose constraints are quadratic Casimir invariant. Applying an SRK method to the extended system preserves these quadratic structures, and hence recovers the original conservative or dissipative energy law on the constraint manifold.

\section{Numerical discretization for CH systems}
We now apply the AIQ framework to CH equations using Fourier pseudo-spectral spatial discretization. The periodic setting provides the discrete summation-by-parts identities needed for the following mass-conservation and energy-dissipation estimates.

The Cahn--Hilliard (CH) equation \cite{cahn1958} describes the dynamics of phase separation, driven by a diffuse-interface free energy. The isotropic CH equation reads
\begin{equation}\label{eq:CH-iso}
\left\{
\begin{aligned}
	&u_t=\nabla\!\cdot\!\big(M\nabla\mu\big),
	\qquad (x,t)\in \Omega \times (0,T],\\
	&\mu=-\Delta u+\varepsilon^{-2}F'(u)+\beta\,\Delta^2 u,
\end{aligned}
\right.
\end{equation}
where \(0<\varepsilon\ll 1\) is the interfacial thickness, \(\beta\ge0\) is the regularization strength, and \(F(u)\) is the bulk potential. Its free energy functional is given by
\begin{equation}\label{eq:CH-iso-e}
E_{\mathrm{iso}}(u)
=\frac12(\nabla u,\nabla u)+\varepsilon^{-2}(F(u),1)+\frac{\beta}{2}(\Delta u,\Delta u),
\end{equation}
where \((\cdot,\cdot)\) denotes the \(L^2(\Omega)\) inner product. namely \((\phi,\psi)=\int_\Omega \phi(x)\psi(x)\,\mathrm{d}x\). It is known that
\begin{align}\label{eq:iso-EnDis}
\frac{\mathrm d}{\mathrm d t}E_{\mathrm{iso}}(u)
=\int_\Omega \mu\,u_t\,\mathrm d x
=\int_\Omega \mu\,\nabla\!\cdot\!\bigl(M\nabla\mu\bigr)\,\mathrm d x
=-M\int_\Omega |\nabla\mu|^2\,\mathrm d x\le 0
\end{align}
which indicates the system is dissipative. Under periodic boundary conditions, the total mass is conserved. Indeed,
\begin{align}\label{eq:iso-MaCon}
\frac{\mathrm d}{\mathrm dt}\int_\Omega u(x,t)\,\mathrm dx
=
\int_\Omega u_t\,\mathrm dx
=
\int_\Omega \nabla\!\cdot\!\bigl(M\nabla\mu\bigr)\,\mathrm dx
=
0,
\end{align}
where the last equality follows from the cancellation of periodic boundary fluxes.

For Eq.\eqref{eq:CH-iso}, we apply the AIQ method in time and the Fourier pseudo-spectral discretization in space. In this section, we show that the resulting numerical discretizations preserve the dissipative property of the CH equation.

Let $\Omega=[a_x,b_x]\times[a_y,b_y]$, and let $h_x$ and $h_y$ be the grid sizes in the $x$ and $y$ directions respectively.
Set $\Omega_h
=
\{(x_j,y_k)| x_j=a_x+jh_x,\;
y_k=a_y+kh_y\}$ and define $V_h=\{U|U=\{U_{jk}\},(x_j,y_k)\in\Omega_h \}$
as the space of grid functions on $\Omega_h$. Denote \(D_x\) and \(D_y\) as the first-order Fourier pseudo-spectral differentiation matrices \cite{Lu2024}. We use the discrete inner product \((U,V)_h=h_xh_y\sum_{j,k}U_{jk}V_{jk}\). For a nonnegative grid weight \(W\), we write \((A,B)_{h,W}:=(W\odot A,B)_h\); in particular, \((A,B)_{h,\Gamma}:=(\Gamma(\Theta)\odot A,B)_h\). On the periodic grid, the Fourier differentiation matrices satisfy the corresponding summation-by-parts identities and the discrete Laplacian has zero mean.
Taking \(M=1\) and $F(u)=\frac14(u^2-1)^2$, we have \(F'(u)=u^3-u\). Employing the Fourier pseudo-spectral method for system~\eqref{eq:CH-iso} in space gives
\begin{equation}\label{eq:CH-iso-semi-dis}
U_t
-\Delta_h\Bigl[
-\Delta_hU+\varepsilon^{-2}(U^{\odot 3}-U)+\beta\Delta_h^2U
\Bigr]=0,
\end{equation} 
for $U\in V_h$. Here, $\Delta_h$  denotes the discrete Laplacian which is  $\Delta_h U = D_x^2U + U(D_y^\top)^2$,  and \(\odot\) denotes the Hadamard product.\footnote{For any $U, V\in V_h$, the Hadamard product defined by $(U\odot V)_{ij}=U_{ij}V_{ij}$, $U^{\odot 2}=U\odot U$.}

By introducing the auxiliary variable \(q=\frac12(u^2-1)\), we apply an $s$-stage SRK method to system~\eqref{eq:CH-iso-semi-dis}. The fully discrete scheme is given by
\begin{equation}\label{fu-dis-scheme:CH-iso}
\left\{
\begin{aligned}
	U_i
	&=
	U^n+\Delta t\sum_{j=1}^s a_{ij}K_j,
	\qquad i=1,\dots,s,\\
	Q_i
	&=
	Q^n+\Delta t\sum_{j=1}^s a_{ij}L_j,
	\qquad i=1,\dots,s,\\
	K_i
	&=
	\Delta_h\mu_i,\\
	\mu_i
	&=
	-\Delta_h U_i
	+2\varepsilon^{-2}U_i\odot Q_i
	+\beta\Delta_h^2U_i,\\
	L_i
	&=
	U_i\odot K_i,\\
	U^{n+1}
	&=
	U^n+\Delta t\sum_{i=1}^s b_iK_i,\\
	Q^{n+1}
	&=
	Q^n+\Delta t\sum_{i=1}^s b_iL_i .
\end{aligned}
\right.
\end{equation}
Here $U_i,Q_i,\mu_i$ are the internal SRK stage values, and $K_i,L_i$
denote the corresponding stage derivatives. The coefficients $(a_{ij},b_i)$
satisfy the symplectic condition~\eqref{eq:srk-symplectic-condition}.

We set
\begin{equation}\label{eq:E-iso-h-rk}
\widetilde E_{\mathrm{iso},h}(U,Q)
=
-\frac12(\Delta_hU,U)_h
+\varepsilon^{-2}\|Q\|_h^2
+\frac{\beta}{2}\|\Delta_hU\|_h^2 .
\end{equation}
The original discrete energy is defined by
\begin{equation}\label{eq:E-iso-h}
E_{\mathrm{iso},h}(U)
=
-\frac12(\Delta_hU,U)_h
+\varepsilon^{-2}
\left\|
\frac12\bigl(U^{\odot2}-\mathbf1\bigr)
\right\|_h^2
+\frac{\beta}{2}\|\Delta_hU\|_h^2.
\end{equation}
If $b_i\ge 0$ for $i=1,\dots,s$, the scheme satisfies the discrete
energy dissipation law with respect to $\widetilde E_{\mathrm{iso},h}$.
Moreover, since the constraint \(Q=\frac12\bigl(U^{\odot2}-\mathbf1\bigr)\) is a quadratic invariant of the extended system, it is preserved by the SRK
discretization. Hence, on the constraint manifold \(\mathcal M_h=\{(U,Q)\in V_h \times V_h|Q_{jk}=\tfrac12(U_{jk}^2-1),\forall j,k\}\) we obtain
\[
\widetilde E_{\mathrm{iso},h}(U,Q)|_{\mathcal M_h}=E_{\mathrm{iso},h}(U).
\]
The discrete energy-dissipation law is as follows
\begin{proposition}\label{thm:CH-iso-e-dis}
	The numerical discretization~\eqref{fu-dis-scheme:CH-iso} satisfies the discrete energy-dissipation law:
	\[
		\frac{E_{\mathrm{iso},h}(U^{n+1})-E_{\mathrm{iso},h}(U^n)}{\Delta t}
		=
		-\sum_{i=1}^s b_i
		\left(
		\|D_x\mu_i\|_h^2+\|\mu_iD_y^\top\|_h^2
		\right).
	\]
	In particular, if \(b_i\ge0\) for \(i=1,\ldots,s\), then
	\[
	E_{\mathrm{iso},h}(U^{n+1})\le E_{\mathrm{iso},h}(U^n).
	\]
\end{proposition}
\begin{proof}
	Since the Fourier pseudo-spectral Laplacian satisfies the discrete integration by parts formula, applying a symplectic Runge--Kutta method to this system gives
	\[
	\begin{aligned}
		&\widetilde E_{\mathrm{iso},h}(U^{n+1},Q^{n+1})
		-
		\widetilde E_{\mathrm{iso},h}(U^n,Q^n)  \\
		=&
		\Delta t\sum_{i=1}^s b_i
		\left[
		\bigl(
		-\Delta_h U_i+\beta\Delta_h^2U_i,K_i
		\bigr)_h
		+
		2\varepsilon^{-2}(Q_i,U_i\odot K_i)_h
		\right]  \\
		=&
		\Delta t\sum_{i=1}^s b_i
		\left[\bigl(
		-\Delta_h U_i
		+2\varepsilon^{-2}U_i\odot Q_i
		+\beta\Delta_h^2U_i,
		K_i
		\bigr)_h\right] \\
		=&
		\Delta t\sum_{i=1}^s b_i(\mu_i,\Delta_h\mu_i)_h .
	\end{aligned}
	\]
	Using again the discrete integration by parts formula yields \((\mu_i,\Delta_h\mu_i)_h=-\|D_x\mu_i\|_h^2-\|\mu_iD_y^\top\|_h^2\). Since \(b_i\ge0\), we obtain 
	\[
	\widetilde E_{\mathrm{iso},h}(U^{n+1},Q^{n+1})
	\le
	\widetilde E_{\mathrm{iso},h}(U^n,Q^n)
	\]
	It remains to connect the modified energy with the original one. For the quadratic constraint
	\(C(U,Q)=Q-\tfrac12(U^{\odot2}-\mathbf1)\), the stage equations give
	\(\dot C=L-U\odot K=0\). Since SRK methods satisfying condition~\eqref{eq:srk-symplectic-condition} preserve quadratic invariants, \(C(U^{n+1},Q^{n+1})=C(U^n,Q^n)\). Therefore, if the initial auxiliary variable is chosen consistently, \(Q^n=\tfrac12((U^n)^{\odot2}-\mathbf1)\) for all $n$. Furthermore, on the constraint manifold we have 
	\[
	E_{\mathrm{iso},h}(U^{n+1})
	\le
	E_{\mathrm{iso},h}(U^n).
	\]
	This completes the proof.
\end{proof}

Anisotropic interfacial energies are used to describe direction-dependent phenomena such as faceted pattern formation and crystal growth. In CH-type models, this directional dependence leads to more complicated energy densities, often with rational-like terms involving the gradient of the phase field~\cite{ma2006implementation}.
The anisotropic CH equation considered here is
\begin{equation}\label{eq:CH-ani}
\left\{
\begin{aligned}
	&u_t = \nabla\!\cdot\!\big(M\nabla\mu\big),
	\qquad (x,t)\in \Omega \times (0,T],\\
	&\mu=
	-\,\nabla\!\cdot\big(\Gamma(\theta)\,\nabla u\big)
	+\frac{\Gamma(\theta)}{\varepsilon^2}F'(u)
	+\beta\,\Gamma(\theta)\,\Delta^2 u,
\end{aligned}
\right.
\end{equation}
where $\Gamma(\theta)$ is the anisotropy factor with $\theta=\operatorname{atan2}(u_y,u_x)$.
The corresponding energy can be written as
\begin{equation}\label{eq:E-ani}
E_{\mathrm{ani}}(u)
=
\frac12 (\nabla u,\nabla u)_\Gamma
+\varepsilon^{-2}(F(u),1)_\Gamma
+\frac{\beta}{2}(\Delta u,\Delta u)_\Gamma,
\end{equation}
where $(\phi,\psi)_\Gamma$ denotes the continuous $\Gamma$-weighted $L^2(\Omega)$ inner product defined by $(\phi,\psi)_\Gamma =\int_\Omega \Gamma(\Theta)\,\phi\,\psi\,\mathrm{d}x$. Its discrete counterpart is denoted by \((\cdot,\cdot)_{h,\Gamma}\), as defined above.
Clearly, the energy is dissipative:
\[
\frac{\rm d}{\rm dt}E_{\mathrm{ani}}(u)=-M\int_\Omega |\nabla\mu|^2\,\rm d x\le 0.
\]
The anisotropic CH equation satisfies the same mass-conservation law as~\eqref{eq:iso-MaCon}.

For $U\in V_h$, denote $\Gamma(\Theta)=\Gamma(\Theta(U_{jk}))$  with $\Theta(U_{jk})=\operatorname{atan2}((U_{jk})_y,(U_{jk})_x)$ the discrete orientation.
Applying the Fourier pseudo-spectral method in space to~\eqref{eq:CH-ani}, yields the semi-discrete numerical discretization
\begin{equation}\label{eq:semi-discrete-ani}
U_t-\Delta_h\Psi(U)=0,
\end{equation}
where $\Psi(U)
=-D_x\,\bigl( \Gamma(\Theta)\odot D_xU\bigr)
-\bigl( \Gamma(\Theta)\odot (UD_y^\top)\bigr)D_y^\top+\varepsilon^{-2} \Gamma(\Theta)\odot f(U)
+\beta\, \Gamma(\Theta)\odot \Delta_h^2U.$
Let the discrete weighted norm be \(|U|_{1,h,\Gamma}^2=
(U_x,U_x)_{h,\Gamma}
+
(U_y,U_y)_{h,\Gamma}\),
and \(\|\Delta_h U\|_{h,\Gamma}^2=(\Delta_h U,\Delta_h U)_{h,\Gamma}\), then the discrete anisotropic energy is defined by
\begin{equation}\label{eq:E-ani-h}
E_{\mathrm{ani},h}(U)
=
\frac12 |U|_{1,h,\Gamma}^2
+\varepsilon^{-2}\bigl(F(U),\,\mathbf 1\bigr)_{h,\Gamma}
+\frac{\beta}{2}\|\Delta_h U\|_{h,\Gamma}^2.
\end{equation}
For the anisotropic case, applying the AIQ method to~\eqref{eq:semi-discrete-ani} gives the following fully discrete scheme:
\begin{equation}\label{fu-dis-scheme:CH-ani}
\left\{
\begin{aligned}
	U_i &= U^n+\Delta t\sum_{j=1}^s a_{ij}K_j,\\
	Y_{\ell,i} &= Y_\ell^n+\Delta t\sum_{j=1}^s a_{ij}P_{\ell,j},
	\qquad \ell=1,\dots,6,\\
	Z_{\ell,i} &= Z_\ell^n+\Delta t\sum_{j=1}^s a_{ij}R_{\ell,j},
	\qquad \ell=1,2,\\
	\Phi_{1,i} &= \Phi_1^n+\Delta t\sum_{j=1}^s a_{ij}S_{1,j},\\
	U^{n+1} &= U^n+\Delta t\sum_{i=1}^s b_iK_i,\\
	Y_\ell^{n+1} &= Y_\ell^n+\Delta t\sum_{i=1}^s b_iP_{\ell,i},
	\qquad \ell=1,\dots,6,\\
	Z_\ell^{n+1} &= Z_\ell^n+\Delta t\sum_{i=1}^s b_iR_{\ell,i},
	\qquad \ell=1,2,\\
	\Phi_1^{n+1} &= \Phi_1^n+\Delta t\sum_{i=1}^s b_iS_{1,i},
\end{aligned}
\right.
\qquad i=1,\dots,s,
\end{equation}
where we choose the fourfold anisotropy function $\Gamma(\theta)=1+\alpha\cos(4\theta) = 1-3\alpha+4\alpha\,\frac{u_x^4+u_y^4}{(u_x^2+u_y^2)^2}$. This is exactly the rational-like function discussed in Example~\ref{eg:fourfold},. We define the same auxiliary variables $y_1,\cdots,y_6$ together with $z_1=\frac12(u^2-1),\,z_2=z_1^2,\,\phi_1=(\Delta u)^2.$ Their discrete counterparts are denoted by \(Y_1,\dots,Y_6\), \(Z_1\), \(Z_2\), and \(\Phi_1\), respectively. The symbols of Eq.~\eqref{fu-dis-scheme:CH-ani} are as follows: \begin{equation}\label{eq:CH-ani-aux-rk-derivatives} 
\left\{ \begin{aligned} K_i &= \Delta_h\mu_i,\\ P_{1,i} &= 2\,U_{x,i}\odot K_{x,i},\\ P_{2,i} &= 2\,U_{y,i}\odot K_{y,i},\\ P_{3,i} &= 2\,(Y_{1,i}+Y_{2,i})\odot(P_{1,i}+P_{2,i}),\\ P_{4,i} &= -\,Y_{4,i}^{\odot 2}\odot P_{3,i},\\ P_{5,i} &= 2\,Y_{1,i}\odot P_{1,i} +2\,Y_{2,i}\odot P_{2,i},\\ P_{6,i} &= Y_{4,i}\odot P_{5,i} +Y_{5,i}\odot P_{4,i},\\ R_{1,i} &= U_i\odot K_i,\\ R_{2,i} &= 2\,Z_{1,i}\odot R_{1,i},\\ S_{1,i} &= 2\,\Delta_h U_i\odot \Delta_h K_i . \end{aligned} \right. 
\end{equation}
For the anisotropic scheme, the stage derivative \(K_i\) has the same conservative form \(K_i=\Delta_h\mu_i\), with the anisotropic chemical potential and auxiliary-variable relations specified in Appendix~A. The mass conservation property therefore follows directly from the discrete divergence form of the stage derivatives and is independent of the auxiliary variables.

Further details are given in Appendix~A. The following proposition  shows that the proposed scheme preserves the energy dissipation.
\begin{proposition}\label{thm:ani-diss-law}
	Assume that the auxiliary variables are initialized consistently, so that	the auxiliary constraints in Appendix~A are satisfied at \(t=0\). Then the numerical discretization~\eqref{fu-dis-scheme:CH-ani} satisfies the discrete	energy-dissipation identity:
	\[
	\begin{aligned}
		\frac{E_{\mathrm{ani},h}(U^{n+1})-E_{\mathrm{ani},h}(U^n)}{\Delta t}
		&=
		-\sum_{i=1}^s b_i
		\left(
		\|D_x\mu_i\|_h^2+\|\mu_iD_y^\top\|_h^2
		\right).
	\end{aligned}
	\]
	In particular, if \(b_i\ge0\) for \(i=1,\ldots,s\), then
	\[
	E_{\mathrm{ani},h}(U^{n+1})\le E_{\mathrm{ani},h}(U^n).
	\]
\end{proposition}
\begin{proof}
	The definitions in Appendix~A imply that the anisotropic discrete energy admits a quadratic modified form
	\[
	\widetilde E_{\mathrm{ani},h}
	=
	\frac12\bigl(\Gamma(Y_6)\odot(Y_1+Y_2),\mathbf1\bigr)_h
	+\varepsilon^{-2}\bigl(\Gamma(Y_6)\odot Z_2,\mathbf1\bigr)_h
	+\frac{\beta}{2}\bigl(\Gamma(Y_6)\odot\Phi_1,\mathbf1\bigr)_h ,
	\]
	where, for the fourfold anisotropy, \(\Gamma(Y_6)=1-3\alpha+4\alpha Y_6\).
	On the constraint manifold we have
	\[
	\widetilde E_{\mathrm{ani},h}=E_{\mathrm{ani},h}(U).
	\]
	Thus, using the SRK quadratic energy identity, we obtain
	\[
	E_{\mathrm{ani},h}(U^{n+1})-E_{\mathrm{ani},h}(U^n)
	=
	\widetilde E_{\mathrm{ani},h}^{\,n+1}
	-
	\widetilde E_{\mathrm{ani},h}^{\,n}
	=
	\Delta t\sum_{i=1}^s b_i(\mu_i,K_i)_h .
	\]
	Since \(K_i=\Delta_h\mu_i\), the discrete summation-by-parts identity gives \((\mu_i,\Delta_h\mu_i)_h	=	-\|D_x\mu_i\|_h^2-\|\mu_iD_y^\top\|_h^2\le0\). Therefore,
	\[
	E_{\mathrm{ani},h}(U^{n+1})-E_{\mathrm{ani},h}(U^n)
	=
	-\Delta t\sum_{i=1}^s b_i
	\left(
	\|D_x\mu_i\|_h^2+\|\mu_iD_y^\top\|_h^2
	\right)
	\le0,
	\]
	where the last inequality follows from \(b_i\ge0\). This proves the discrete
	energy-dissipation law for the anisotropic scheme.
\end{proof}
\begin{proposition}\label{prop:discrete-mass-conservation}
	Assume that the Fourier pseudo-spectral Laplacian on the periodic grid satisfies \((\Delta_h V,\mathbf1)_h=0\) for all \(V\in V_h\). Then both the isotropic AIQ scheme~\eqref{fu-dis-scheme:CH-iso} and the anisotropic AIQ scheme~\eqref{fu-dis-scheme:CH-ani} preserve the discrete mass:
	\[
	(U^{n+1},\mathbf1)_h=(U^n,\mathbf1)_h .
	\]
\end{proposition}
	
\begin{proof}
	For both schemes, the update of the physical variable is
	\[
	U^{n+1}=U^n+\Delta t\sum_{i=1}^s b_iK_i,
	\qquad K_i=\Delta_h\mu_i .
	\]
	Taking the discrete inner product with \(\mathbf1\) gives
	\[
	(U^{n+1},\mathbf1)_h-(U^n,\mathbf1)_h
	=
	\Delta t\sum_{i=1}^s b_i(\Delta_h\mu_i,\mathbf1)_h=0.
	\]
	Thus the discrete mass is preserved for arbitrary SRK coefficients \(b_i\). In particular, the anisotropic chemical potential changes the value of \(\mu_i\), but not the identity \((\Delta_h\mu_i,\mathbf1)_h=0\) imposed by periodic summation by parts.
\end{proof}
	
\section{Dispersion analysis and physical properties}
We now apply the AIQ framework to the CH systems and examine their analytical properties. For the CH equation, the bulk potential \(F(u)\) usually takes a double-well form, e.g.,
$$
F(u)=\tfrac14(u^2-1)^2.
$$
Other common choices include the Flory Huggins form \cite{flory1953principles, huggins1941solutions}
$$
F(u)=(1+u)\ln(1+u)+(1-u)\ln(1-u)+\theta u^2
$$
and the form of higher-order polynomial \cite{wise2009energy}
$$
F(u)=\tfrac{a}{2}u^2+\tfrac{b}{4}u^4+\tfrac{c}{6}u^6 .
$$
The three commonly used potential functions are shown in Fig.~\ref{fig:potential}. The red dot in each panel marks the inflection point of the corresponding potential.
\begin{figure}[htbp]
\centering
\begin{subfigure}[t]{0.33\textwidth}
	\centering
	\includegraphics[width=\linewidth]{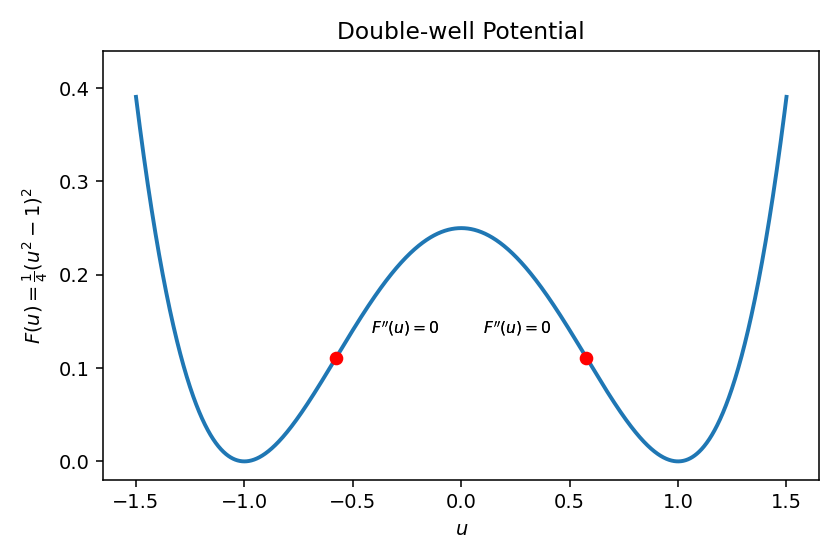}
	\caption{Double-well potential}
\end{subfigure}\hfill
\begin{subfigure}[t]{0.33\textwidth}
	\centering
	\includegraphics[width=\linewidth]{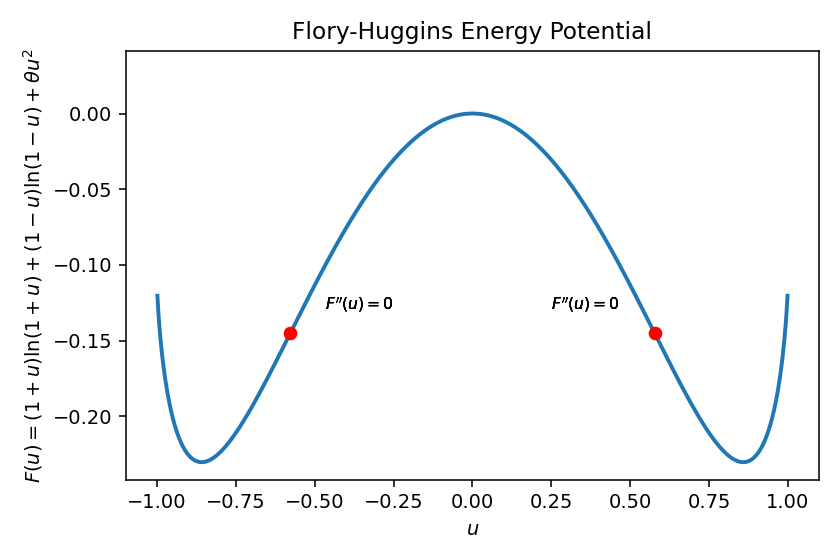}
	\caption{Flory--Huggins potential}
\end{subfigure}\hfill
\begin{subfigure}[t]{0.33\textwidth}
	\centering
	\includegraphics[width=\linewidth]{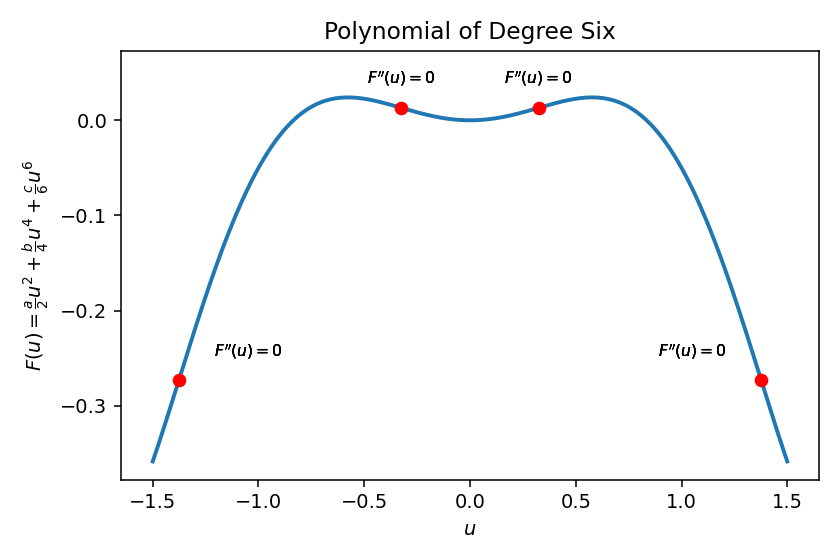}
	\caption{Sixth-order polynomial potential}%
\end{subfigure}
\caption{Three commonly used potential functions.}
\label{fig:potential}
\end{figure}
In what follows, we focus primarily on the double-well potential, while other choices are included to illustrate the broader range of admissible free-energy densities in the CH framework.

\paragraph{Dispersion analysis for the isotropic CH equation}
To derive the dispersion relation, we investigate the evolution of small-amplitude perturbations
around a homogeneous state by setting $u(\mathbf x,t)=u_0+\eta(\mathbf x,t)$ with $|\eta|\ll 1.$
Since the mass is conserved, $u_0=\frac{1}{|\Omega|}\int_\Omega u(\mathbf x,0)\,\mathrm{d}\mathbf x$ is constant. Linearizing CH system~\eqref{eq:CH-iso} around \(u_0\) with constant mobility gives
\begin{equation}\label{eq:CH-linear}
\eta_t
=
M\,\Delta\Bigl(-\Delta\eta+\varepsilon^{-2}F''(u_0)\eta+\beta\Delta^2\eta\Bigr).
\end{equation}

For a Fourier mode $\eta(\mathbf x,t)=\hat\eta_{\mathbf k}(t)e^{i\mathbf k\cdot\mathbf x}$ on a periodic domain, the linearized equation~\eqref{eq:CH-linear} reduces to
\begin{equation}\label{eq:eta}
\frac{\mathrm d}{\mathrm dt}\hat\eta_{\mathbf k}(t)
=
\lambda(\mathbf k;u_0)\,\hat\eta_{\mathbf k}(t),
\end{equation}
where $\lambda(\mathbf k;u_0)=-M\,|\mathbf k|^2\Bigl(|\mathbf k|^2+\varepsilon^{-2}F''(u_0)+\beta|\mathbf k|^4\Bigr)$ denotes the growth rate and  $\mathbf k=(k_x,k_y)$ is the wavevector. If we set $k=|\mathbf k|$ then the dispersion relation can be written as
\[
\lambda(k;u_0)
=
-Mk^2\Bigl(k^2+\varepsilon^{-2}F''(u_0)+\beta k^4\Bigr).
\]

The system~\eqref{eq:eta} is unstable if and only if the growth rate satisfies $\lambda(k;u_0)>0$, i.e., \(k^{2}+\varepsilon^{-2}F''(u_0)+\beta k^{4}<0\). Letting $s=k^{2}\ge0$ implies a  quadratic inequality \(\beta s^{2}+s+\varepsilon^{-2}F''(u_0)<0,\)  which  holds  only when $F''(u_0)<0$.
Moreover, for $\beta>0$ the unstable wavenumbers satisfy
$$
0<s<s_{+},\qquad
s_{+}=\frac{-1+\sqrt{\,1-4\beta\,\varepsilon^{-2}F''(u_0)\,}}{2\beta},
$$
or equivalently,
$$
0<|k|<\sqrt{s_{+}}
=\left(\frac{-1+\sqrt{\,1-4\beta\,\varepsilon^{-2}F''(u_0)\,}}{2\beta}\right)^{1/2}.
$$
In the limit $\beta\to0^{+}$, the above condition reduces to
$0<|k|<\varepsilon^{-1}\sqrt{-F''(u_0)}$.
Therefore, two regimes arise:
\begin{itemize}
\item[(1)] If $F''(u_0)>0$, the system is stable; all perturbations decay.
\item[(2)] If $F''(u_0)<0$, the system is spinodally unstable; small perturbations grow exponentially.
\end{itemize}

For the double-well potential, one has \(F''(u_0)=3u_0^2-1\). Hence the spinodal interval is \(|u_0|<1/\sqrt{3}\). In this regime, the linear growth rate \(\lambda=\lambda(\mathbf{k};u_0)\) is positive on a finite band of wave numbers and nonpositive otherwise. For fixed \(u_0\), the dependence of \(\lambda(\mathbf{k};u_0)\) on \(\mathbf{k}\) is illustrated in Fig.~\ref{fig:dispersion-iso}. Here \(u_0=0\) is chosen in the spinodal region, i.e. \(F''(u_0)<0\), so that the growth rate is positive on a finite band of wave numbers.
\begin{figure}[htbp]
\centering
\begin{subfigure}[t]{0.5\textwidth}
	\centering
	\includegraphics[width=\linewidth]{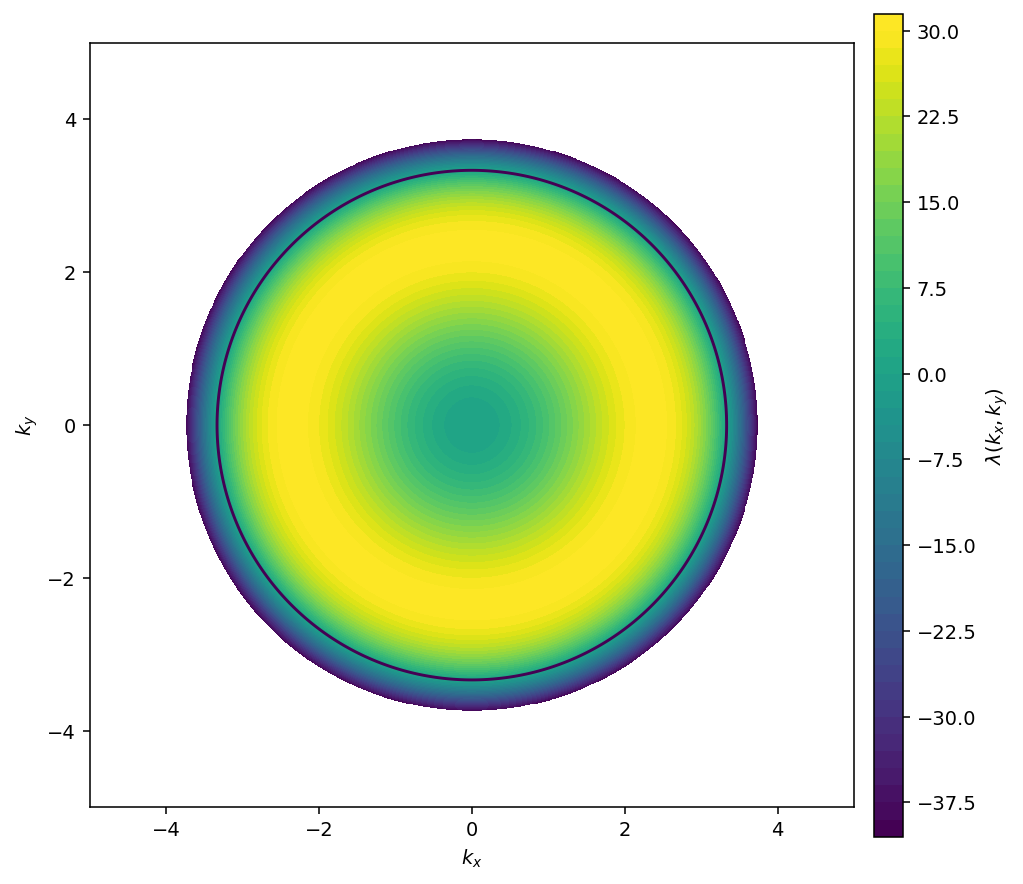}
	\caption{Contour of $\lambda(\mathbf{k})$}
\end{subfigure}\hfill
\begin{subfigure}[t]{0.45\textwidth}
	\includegraphics[width=\linewidth]{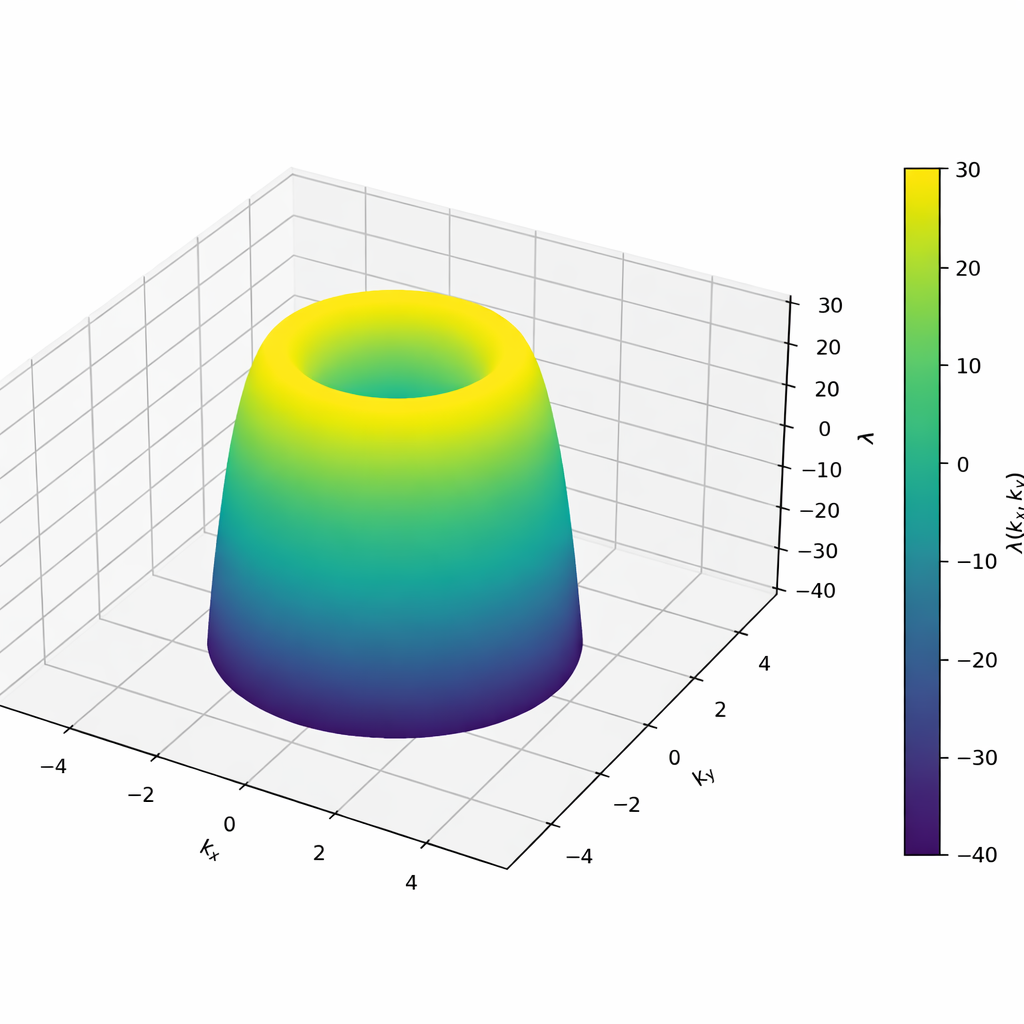}
	\caption{3D surface of $\lambda(\mathbf{k})$}
\end{subfigure}
\caption{Dispersion relation of the linearized isotropic Cahn--Hilliard equation for \(u_0=0\). }
\label{fig:dispersion-iso}
\end{figure}

Consider the following system with periodic boundary conditions
\begin{equation}\label{eq:linear-model}
u_t = a\,\Delta u + b\,\Delta^2 u + c\,\Delta^3 u
\qquad \text{in } \Omega \times (0,\infty),
\end{equation}
where \(a,b,c\in\mathbb{R}\) and \(\Omega=[0,L]^2\). Note that the linearized isotropic CH equation~\eqref{eq:CH-linear} corresponds to \eqref{eq:linear-model} when $a=-M\epsilon^{-2}F''(u_0), b=-M, c=M\beta.$ Then we have the following proposition.
\begin{proposition}\label{proposition:disp_cont_disc_refined}
	Consider Eq.~\eqref{eq:linear-model}. It admits normal-mode solutions of the form $u(\bm x,t)=U_0\,e^{\lambda t}e^{\mathrm i \mathbf k\cdot \bm x}$, $\mathbf k\in\frac{2\pi}{L}\mathbb Z^2$, \(U_0\in\mathbb C\) where the corresponding growth rate \(\lambda\) is given by
	\begin{equation}\label{eq:dis}
		\lambda(\mathbf k)=-a\,k^2+b\,k^4-c\,k^6.
	\end{equation}
\end{proposition}
\begin{proof}
	Substituting $u(\bm x,t)=U_0\,e^{\lambda t}\,e^{\mathrm i \mathbf k\cdot \bm x}$ into \eqref{eq:linear-model} yields
	\[
	u_t=\lambda u.
	\] 
	For the spatial operators, the eigenfunction property of the Fourier mode gives $\Delta u = -k^2 u$, and consequently $\Delta^2 u = k^4 u$ and $\Delta^3 u = -k^6 u$. Inserting these expressions into the PDE leads to the characteristic equation
	\[
	\lambda(\mathbf k)=-a\,k^2+b\,k^4-c\,k^6.
	\]
	This completes the proof.
\end{proof}

In this paper, we discretize the spatial variables using a Fourier spectral method. For system~\eqref{eq:linear-model}, this yields an ODE for each Fourier mode $\hat{\eta}(\mathbf{k},t)$:
\begin{equation}
\frac{d}{dt}\hat{\eta}(\mathbf{k},t) = \lambda(\mathbf{k})\hat{\eta}(\mathbf{k},t),
\end{equation}
where the growth rate $\lambda(\mathbf{k})$ is given by~\eqref{eq:dis}.
\begin{proposition}\label{proposition:DispDisc}
Let $\Delta t$ denote the time step. Applying  Runge--Kutta method to system~\eqref{eq:eta} yields
\begin{align}\label{eq:DisDis}
	e^{\Delta t \tilde{\lambda}(\mathbf{k})} = R\bigl(\Delta t \lambda(\mathbf{k})\bigr),
\end{align}
where $\tilde{\lambda}(\mathbf{k})$ is the discrete growth rate and $R(z)$ is the stability function.
\end{proposition}
\begin{proof}
	Applying SRK to system~\eqref{eq:eta} leads to
	\begin{align}\label{eq:DisSemiODE}
		{\hat u}^{n+1}=R(\Delta t\lambda(\mathbf{k})){\hat u}^{n}.
	\end{align}
	Upon introducing the discrete growth rate $\tilde{\lambda}$ via ${\hat u}^{n}=U_0\exp(\Delta t \tilde{\lambda})$, Eq.~\eqref{eq:DisSemiODE} reduces to Eq.~\eqref{eq:DisDis}. This completes the proof of this proposition.
\end{proof}

It follows from Proposition~\ref{proposition:DispDisc} that the discrete amplification factors for explicit Euler ($\rm EE$), implicit Euler ($\rm IE$), and implicit midpoint methods ($\rm IM$) can be calculated by
\begin{equation}\label{eq:g_three_methods}
\begin{aligned}
	g_{\rm EE}(\mathbf{k})
	&=1-\Delta t\,M|\mathbf{k}|^2\Bigl(A_0+|\mathbf{k}|^2+\beta|\mathbf{k}|^4\Bigr),\\[0.5ex]
	g_{\rm IE}(\mathbf{k})
	&=\frac{1}{1+\Delta t\,M|\mathbf{k}|^2\Bigl(A_0+|\mathbf{k}|^2+\beta|\mathbf{k}|^4\Bigr)},\\[0.5ex]
	g_{\rm IM}(\mathbf{k})
	&=\frac{1-\frac{\Delta t}{2}M|\mathbf{k}|^2\Bigl(A_0+|\mathbf{k}|^2+\beta|\mathbf{k}|^4\Bigr)}
	{1+\frac{\Delta t}{2}M|\mathbf{k}|^2\Bigl(A_0+|\mathbf{k}|^2+\beta|\mathbf{k}|^4\Bigr)}.
\end{aligned}
\end{equation}
The exact amplification factor is $g_{\rm ex}(\mathbf{k}) = \exp(\Delta t \lambda(\mathbf{k}))$. The discrete amplification maps produced by the three schemes, together with the exact amplification map, are shown in Fig.~\ref{fig:comparison-iso}, where the solid curve indicates the contour level $g(\mathbf{k}) = 1$.
\begin{figure}[htbp]
\centering
\resizebox{0.7\textwidth}{!}{%
	\begin{minipage}{\textwidth}
		\centering
		
		\begin{subfigure}{0.49\linewidth}
			\centering
			\includegraphics[width=\linewidth]{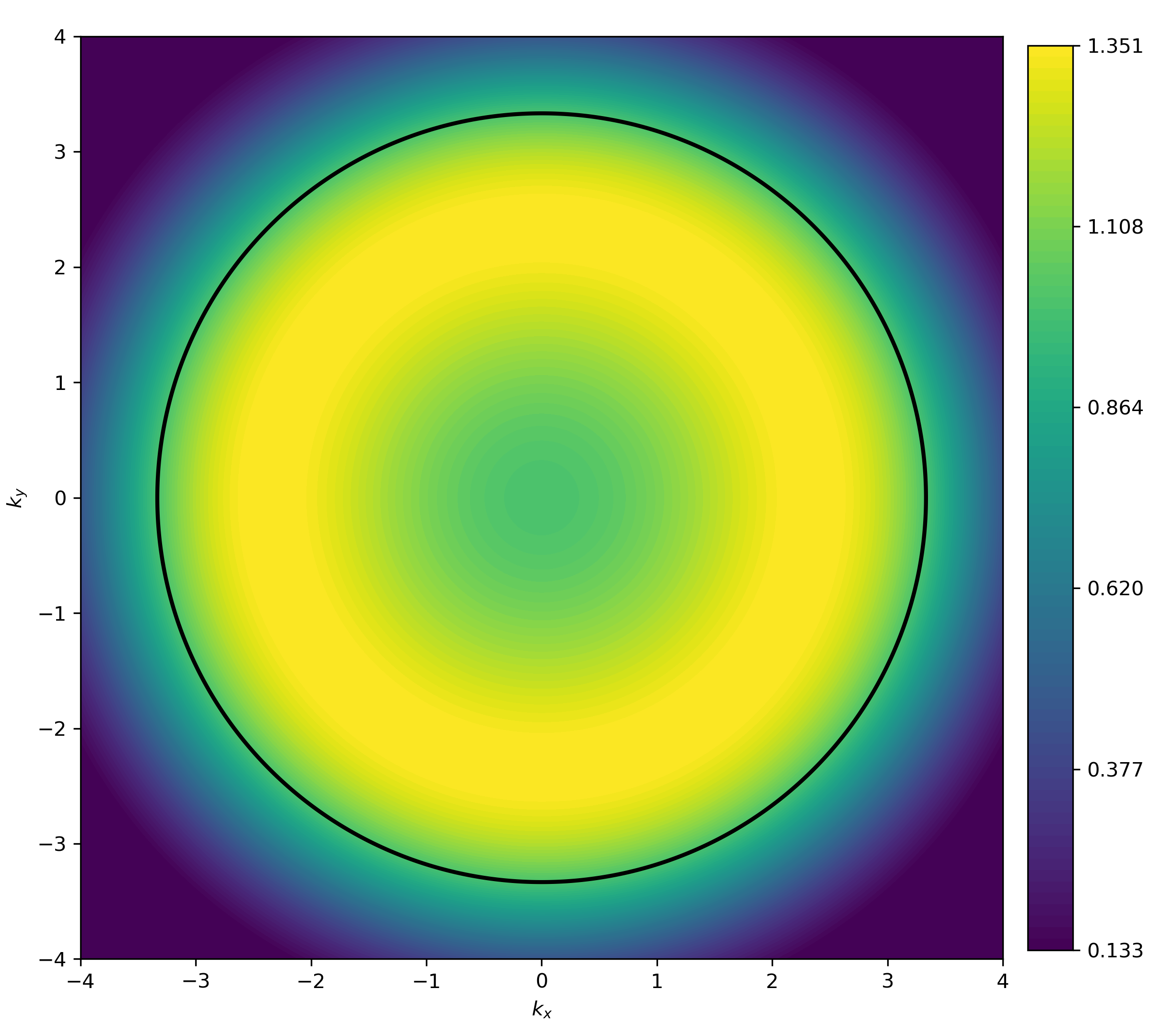}
		\end{subfigure}\hfill
		\begin{subfigure}{0.49\linewidth}
			\centering
			\includegraphics[width=\linewidth]{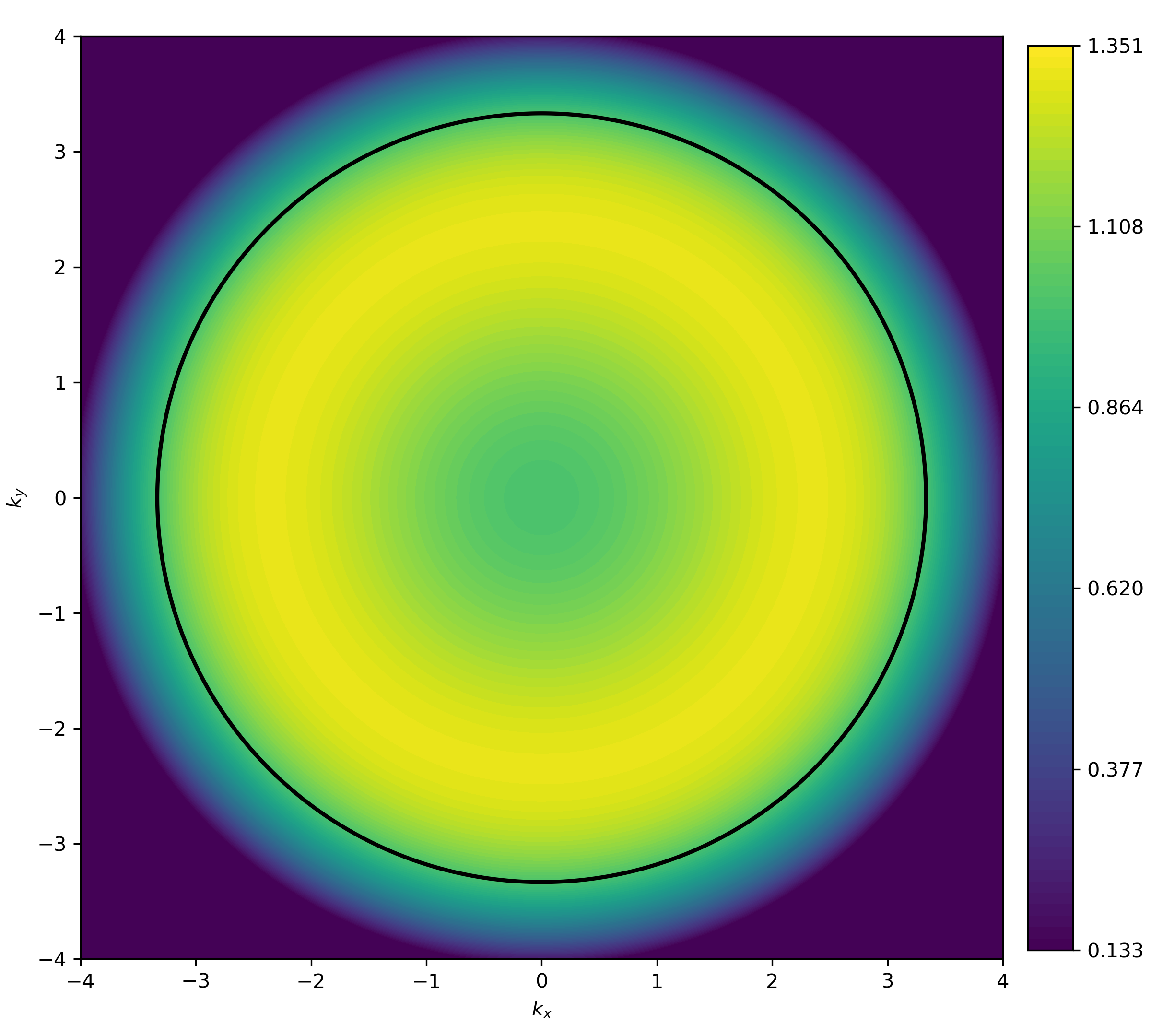}
		\end{subfigure}
		
		\vspace{0.6em}
		
		\begin{subfigure}{0.49\linewidth}
			\centering
			\includegraphics[width=\linewidth]{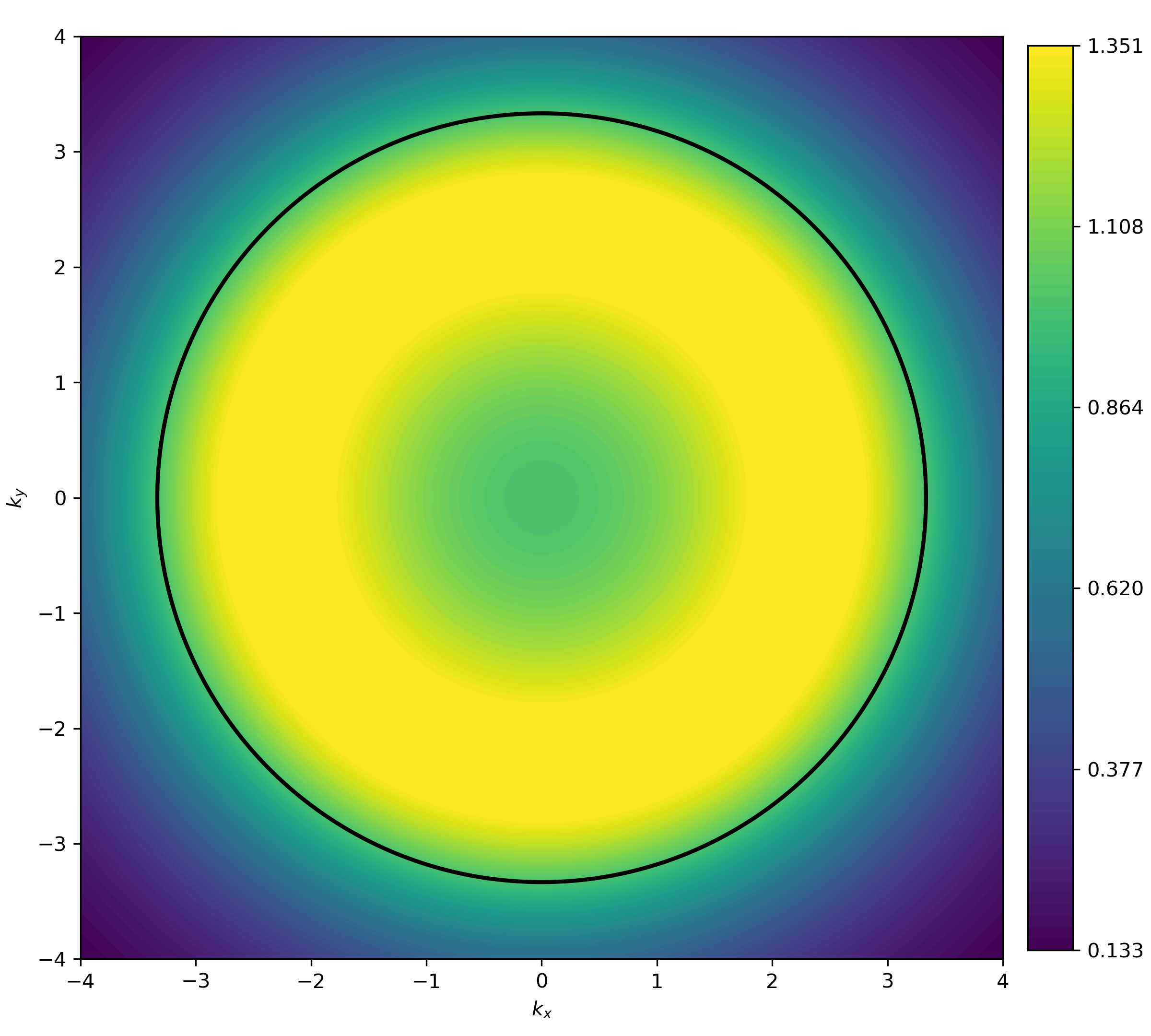}
		\end{subfigure}\hfill
		\begin{subfigure}{0.49\linewidth}
			\centering
			\includegraphics[width=\linewidth]{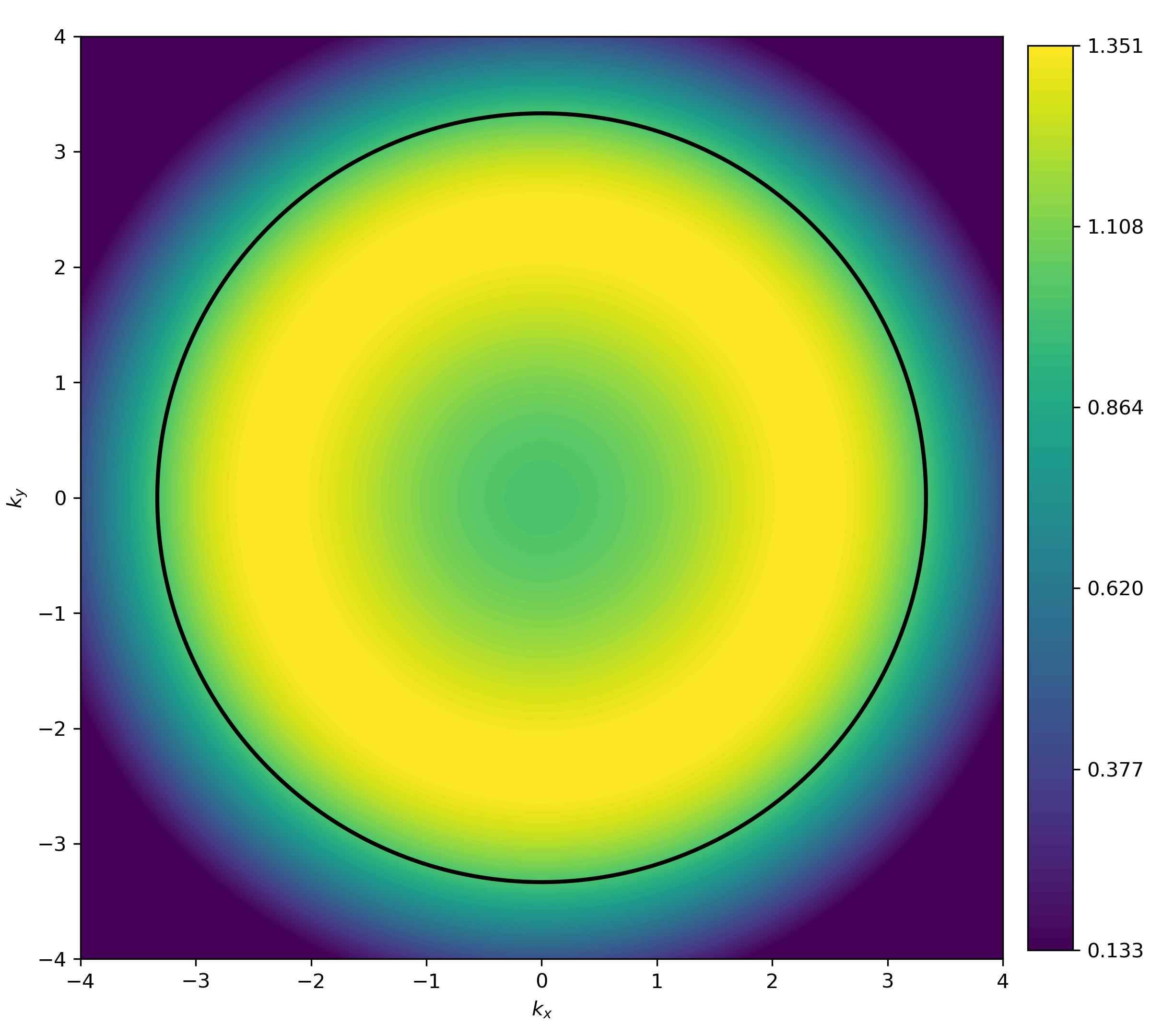}
		\end{subfigure}
		
	\end{minipage}%
	}
\caption{Contours of \(g(\mathbf{k})\) in the \((k_x,k_y)\)-plane for the isotropic case. Top left: exact; top right: explicit Euler; bottom left: implicit Euler; bottom right: implicit midpoint.}
\label{fig:comparison-iso}
\end{figure}
As shown in Fig.~\ref{fig:comparison-iso}, the explicit Euler scheme may produce  $|g_{\rm EE}(\mathbf{k})|>1$ for large wavenumbers $|\mathbf{k}|$, causing high-frequency modes to become unstable. In contrast, the implicit Euler scheme is stable for any $\Delta t$, but it excessively damps the solution. The implicit midpoint scheme is a much closer match to the exact map, providing a second-order accurate and A-stable approximation. As shown in the plot, it better preserves the shape of the level sets in the $(k_x,k_y)$ plane.

\paragraph{Spinodal instability}
From the dispersion relation derived above, we obtain the growth rate of each Fourier mode. This relation allows us to analyze the dependence of $\lambda(\mathbf{k})$ on the homogeneous background state $u_0$, and hence to distinguish the spinodal regime from the stable regime. Spinodal instability occurs when a band of wavenumbers yields $\lambda(\mathbf{k})>0$. This causes the growth of small perturbations.

We consider three homogeneous background states, namely $u_0 = 0,\ 0.70,\ 0.95$, and set the initial data as $\bar u = u_0 + \delta \xi$ with $\xi \sim \mathcal{N}(0,1)$ at each grid point.

Fig.~\ref{fig:spinodal} illustrates the spinodal instability criterion for the isotropic CH model. The homogeneous state is linearly unstable if and only if $F''(u_0) < 0$. In our tests, $u_0 = 0$ yields $F''(u_0) = -1 < 0$; consequently, random perturbations are amplified, and phase separation patterns emerge. In contrast, for $u_0 = 0.70$ and $u_0 = 0.95$, we have $F''(u_0) > 0$, so $\lambda(k) \le 0$ for all $k$. Hence, the perturbations decay and the solution remains nearly homogeneous.

\begin{figure}[htbp]
\centering
\includegraphics[width=\linewidth]{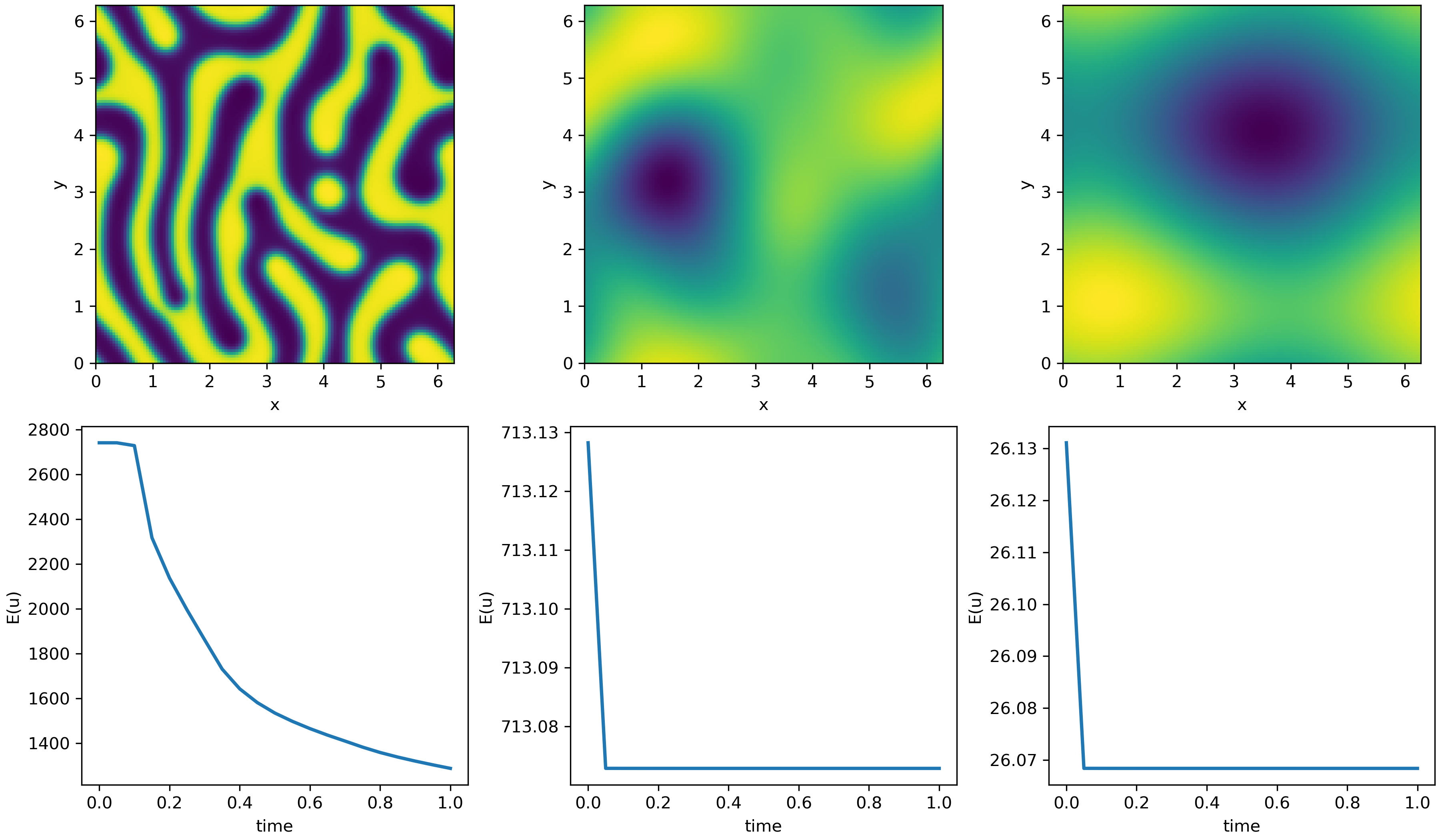}
\caption{Phase separation and energy decay with different initial conditions.}
\label{fig:spinodal}
\end{figure}

\paragraph{Coarsening law for the Isotropic CH model} 
In multi-phase dynamical systems, coarsening is commonly described by a
power-law relation for the characteristic length scale. Below, we present the standard scaling law for the isotropic CH model~\eqref{eq:CH-iso}.

\begin{proposition}[Coarsening scaling law]\label{theorem:coarsening_iso_CH}
Consider the isotropic CH equation~\eqref{eq:CH-iso} on a periodic domain $\Omega$ with constant mobility $M>0$. Let $\bar E(t)=\frac{E(t)}{|\Omega|}$ denote the energy density.  In the late-stage coarsening regime, the dynamics are governed by a single length scale $R(t)$, and the energy is dominated by interfacial contributions, yielding the scaling law
\[
\bar E(t)\sim \frac{\sigma}{R(t)},
\]
where \(\sigma>0\) denotes the surface tension. Then
\[
R(t)\sim (M\sigma t)^{1/3},
\qquad
\bar E(t)\sim M^{-1/3}\sigma^{2/3}t^{-1/3}.
\]
\end{proposition}
\begin{proof}
Differentiating the scaling relation $\bar E(t)\sim \sigma/R(t)$ leads to
\begin{align}\label{eq:1}
	\frac{\mathrm d \bar E}{\mathrm d t}\sim -\frac{\sigma}{R(t)^2}\,\frac{\mathrm d R}{\mathrm d t}.
\end{align}
The Gibbs-Thomson curvature relation gives \(\mu\sim \sigma\kappa\), and under the single-length-scale hypothesis \(\kappa\sim 1/R(t)\), hence \(\mu\sim \sigma/R(t)\). Because $\mu$ varies over a length scale of order $R(t)$ in the late-stage regime, it follows that $|\nabla\mu| \sim \frac{\mu}{R(t)} \sim \frac{\sigma}{R(t)^2}$.
Thus, the energy dissipation per unit volume satisfies
\begin{align}\label{eq:2}
	-\frac{\mathrm d \bar E}{\mathrm d t}
	=\frac{M}{|\Omega|}\int_\Omega |\nabla\mu|^2\,\mathrm d x
	\sim M\Bigl(\frac{\sigma}{R(t)^2}\Bigr)^2.
\end{align}
Combining Eqs.~\eqref{eq:1} and~\eqref{eq:2} yields 
$$
\frac{\mathrm d R}{\mathrm d t}\sim \frac{M\sigma}{R(t)^2}.
$$
Integrating this scaling ODE gives $R(t)\sim (M\sigma\,t)^{1/3}$, and substituting back into $\bar E(t)\sim \sigma/R(t)$ obtains  $\bar E(t)\sim M^{-1/3}\sigma^{2/3}t^{-1/3}.$ This completes the scaling argument.
\end{proof}

We take a random initial condition on \(\Omega=[0,2\pi]^2\), with values uniformly distributed in \([-1,1]\). We simulate system~\eqref{eq:CH-iso} with $\varepsilon=0.3$ and $M=1$ using the proposed AIQ method. Figure~\ref{fig:coarsening-rate} shows that the energy density \(\bar E(t)\) exhibits the predicted \(t^{-1/3}\) decay, and the fitted line has slope \(-1/3\) in the log--log plot.

\begin{figure}[htbp]
\centering
	\includegraphics[width=0.7\linewidth]{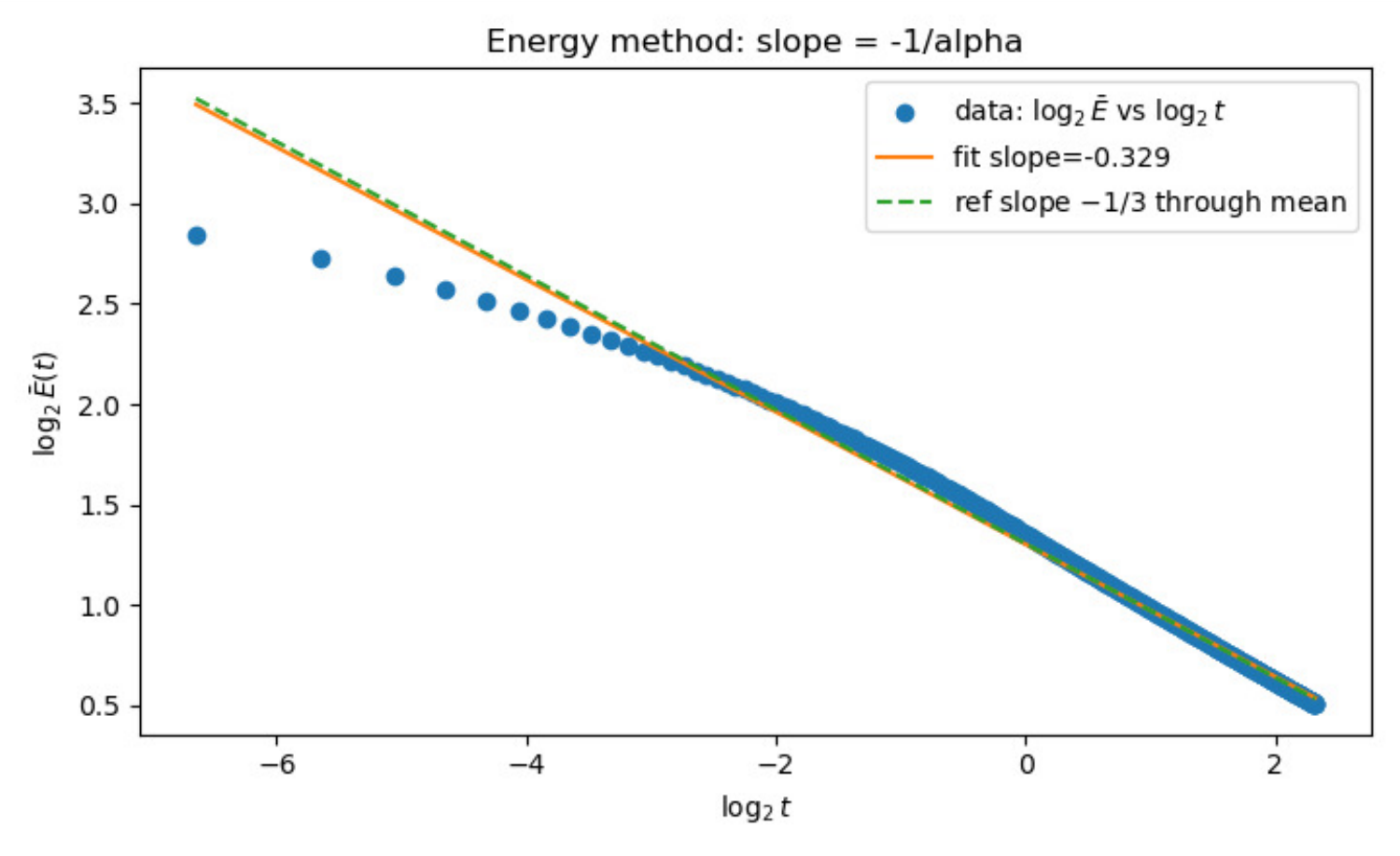}
\caption{Coarsening rate of the isotropic CH equation with $M=1$.}
\label{fig:coarsening-rate}
\end{figure}

\paragraph{Dispersion analysis for the  anisotropic CH equation}
In many crystalline applications, isotropic surface energy proves inadequate for describing faceted morphologies and directionally preferred growth. This motivates the investigation of the anisotropic CH model~\eqref{eq:CH-ani} and its associated dispersion behavior.

Similar to the isotropic case, we linearize the anisotropic model about a homogeneous state. Each Fourier mode $\hat{\eta}(\mathbf{k},t)$ satisfies the scalar ODE~\eqref{eq:eta}, where  
\begin{equation}\label{eq:disp-ani}
\lambda(\mathbf k)
= -M k^2 \Gamma(\theta_{\mathbf{k}})
\bigl(k^2 + \varepsilon^{-2}F''(u_0) + \beta k^4\bigr),
\qquad
\Gamma(\theta_{\mathbf{k}})
=1+\alpha\,\frac{k_x^4-6k_x^2k_y^2+k_y^4}{(k_x^2+k_y^2)^2}.
\end{equation}

Since the anisotropy factor $\Gamma(\theta_{\mathbf{k}})$ is nonnegative, the zeros of $\lambda(\mathbf{k})$ are determined by $k^2 + \varepsilon^{-2}F''(u_0) + \beta k^4 = 0$. In particular, for $\beta = 0$ equation~\eqref{eq:disp-ani} reduces to
$$
\lambda(\mathbf k)=-k^2\,\Gamma(\theta_{\mathbf k})\bigl(k^2+\varepsilon^{-2}F''(u_0)\bigr),
$$
which gives  the spinodal criterion for $F''(u_0)<0$.

If $\beta>0$, setting $s=k^2\ge 0$ gives
$$
\lambda((\mathbf k))=-s\Gamma(\theta_{\mathbf k})\bigl(\varepsilon^{-2}F''(u_0)+s+\beta s^2\bigr).
$$
When $F''(u_0)>0$, one has $\lambda<0$ for all $k$, hence in this case the system is linearly stable. If $F''(u_0)<0$, there exists a positive root $s_+>0$ given by
$$
s_{+}=
\frac{-1+\sqrt{1-4\beta\,\varepsilon^{-2}F''(u_0)}}{2\beta}.
$$

Thus, the system is unstable for \(0<k<\sqrt{s_+}\), which is similar to the isotropic case. For visualization, Fig.~\ref{fig:dispersion-ani} shows a representative directional slice with \(\theta=0\) and \(u_0=0\). Since anisotropy breaks rotational invariance, this slice is used only as a reference direction. Here \(u_0\) is chosen in the spinodal region, so that the unstable band of wave numbers is clearly visible.
\begin{figure}[htbp]
\centering
\begin{subfigure}[t]{0.5\textwidth}
	\centering
	\includegraphics[width=\linewidth]{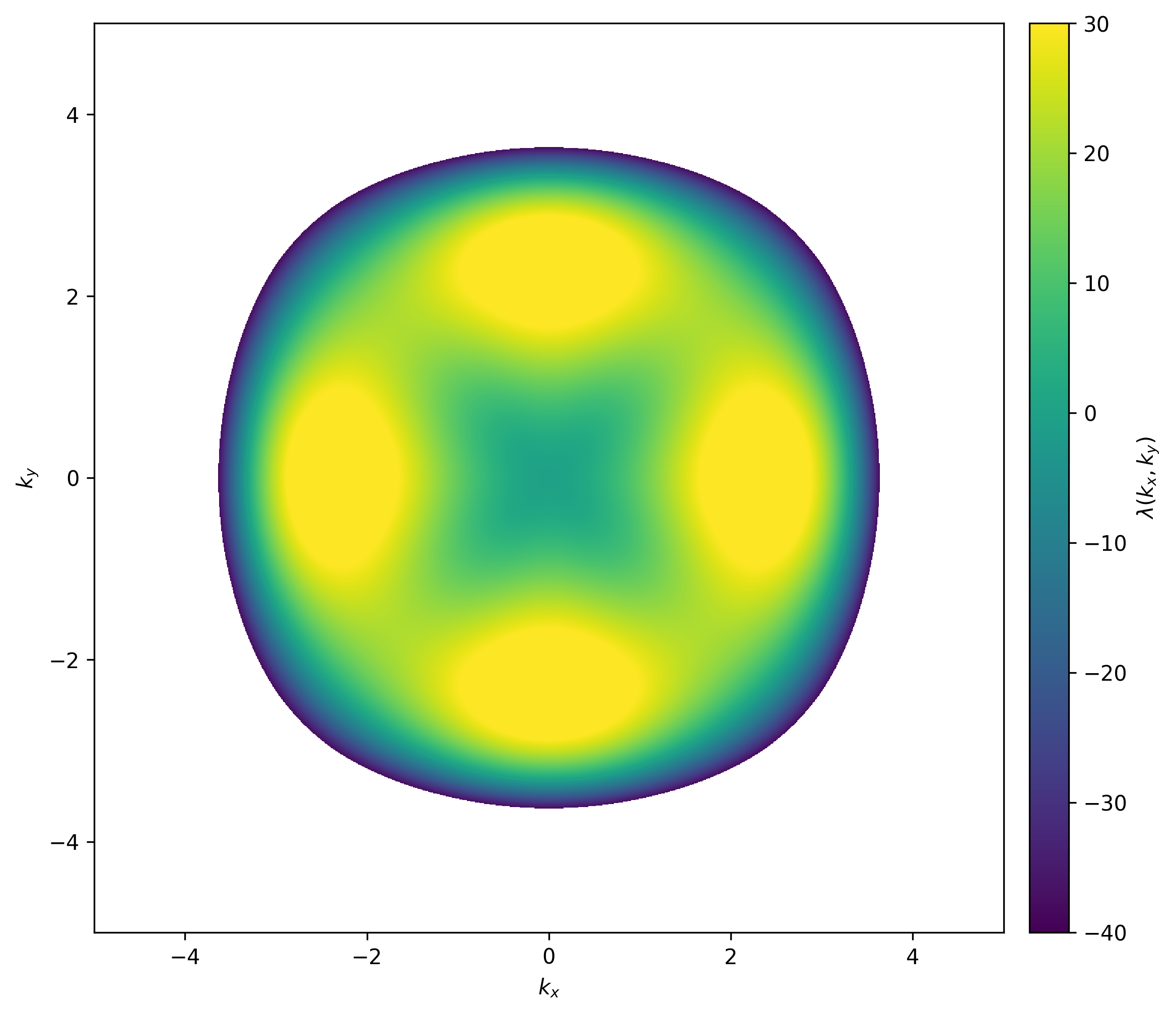}
\end{subfigure}\hfill
\begin{subfigure}[t]{0.45\textwidth}
	\includegraphics[width=\linewidth]{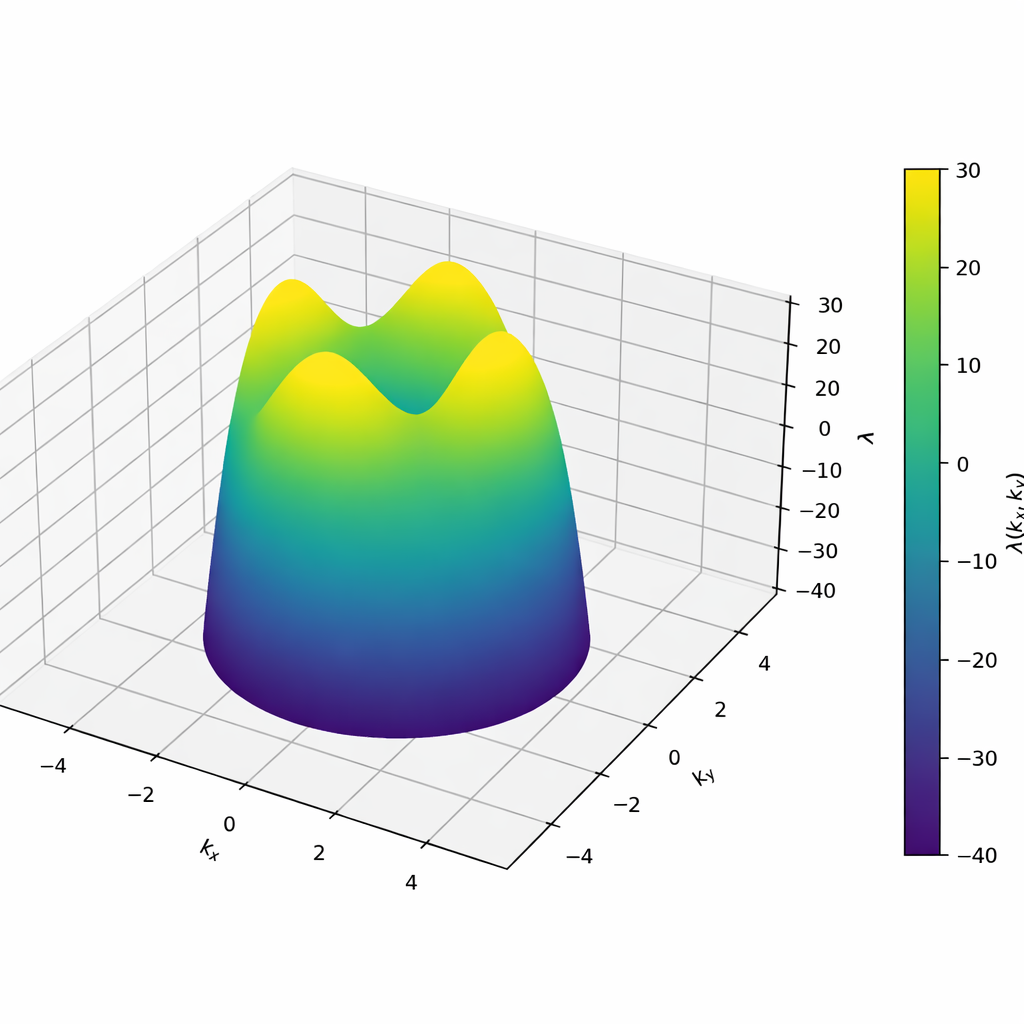}
\end{subfigure}
\caption{Dispersion relation of the linearized anisotropic CH equation with $u_0=0$ and $\theta=0$. Left: Contour plot of $\lambda(\mathbf{k})$;
	Right: 3D surface of $\lambda(\mathbf{k})$. }
\label{fig:dispersion-ani}
\end{figure}

Similar to the isotropic case, the equation~\eqref{eq:eta} yields the exact amplification factor $g_{\rm ex}(\mathbf{k})=\exp\!\bigl(\Delta t\,\lambda(\mathbf{k})\bigr)$.
Applying explicit Euler (EE), implicit Euler (IE), and implicit midpoint (IM) to this mode ODE yields the discrete amplification factors

\begin{equation}\label{eq:g_aniso_three_methods} 
\begin{aligned} 
	g_{\rm EE}(\mathbf k) &=1-\Delta t\,M k^2 \Gamma(\theta_{\mathbf{k}}) \bigl(k^2+\varepsilon^{-2}F''(u_0)+\beta k^4\bigr),\\[0.5ex] g_{\rm IE}(\mathbf k) &=\frac{1}{1+\Delta t\,M k^2 \Gamma(\theta_{\mathbf{k}}) \bigl(k^2+\varepsilon^{-2}F''(u_0)+\beta k^4\bigr)},\\[0.5ex] g_{\rm IM}(\mathbf k) &=\frac{1-\frac{\Delta t}{2}M k^2 \Gamma(\theta_{\mathbf{k}}) \bigl(k^2+\varepsilon^{-2}F''(u_0)+\beta k^4\bigr)} {1+\frac{\Delta t}{2}M k^2 \Gamma(\theta_{\mathbf{k}}) \bigl(k^2+\varepsilon^{-2}F''(u_0)+\beta k^4\bigr)}.
\end{aligned} 
\end{equation}
In Fig.~\ref{fig:comparison-aniso}, we compare the exact amplification factor \(g_{\rm ex}\) with the numerical factors \(g_{\rm EE}\), \(g_{\rm IE}\), and \(g_{\rm IM}\) on the \((k_x,k_y)\) plane. The results show that the implicit midpoint rule better preserves the directional level-set structure of the exact amplification factor than the implicit Euler method.
\begin{figure}[htbp]
	\centering
	\resizebox{0.7\textwidth}{!}{%
		\begin{minipage}{\textwidth}
			\centering
			
			\begin{subfigure}{0.49\linewidth}
				\centering
				\includegraphics[width=\linewidth]{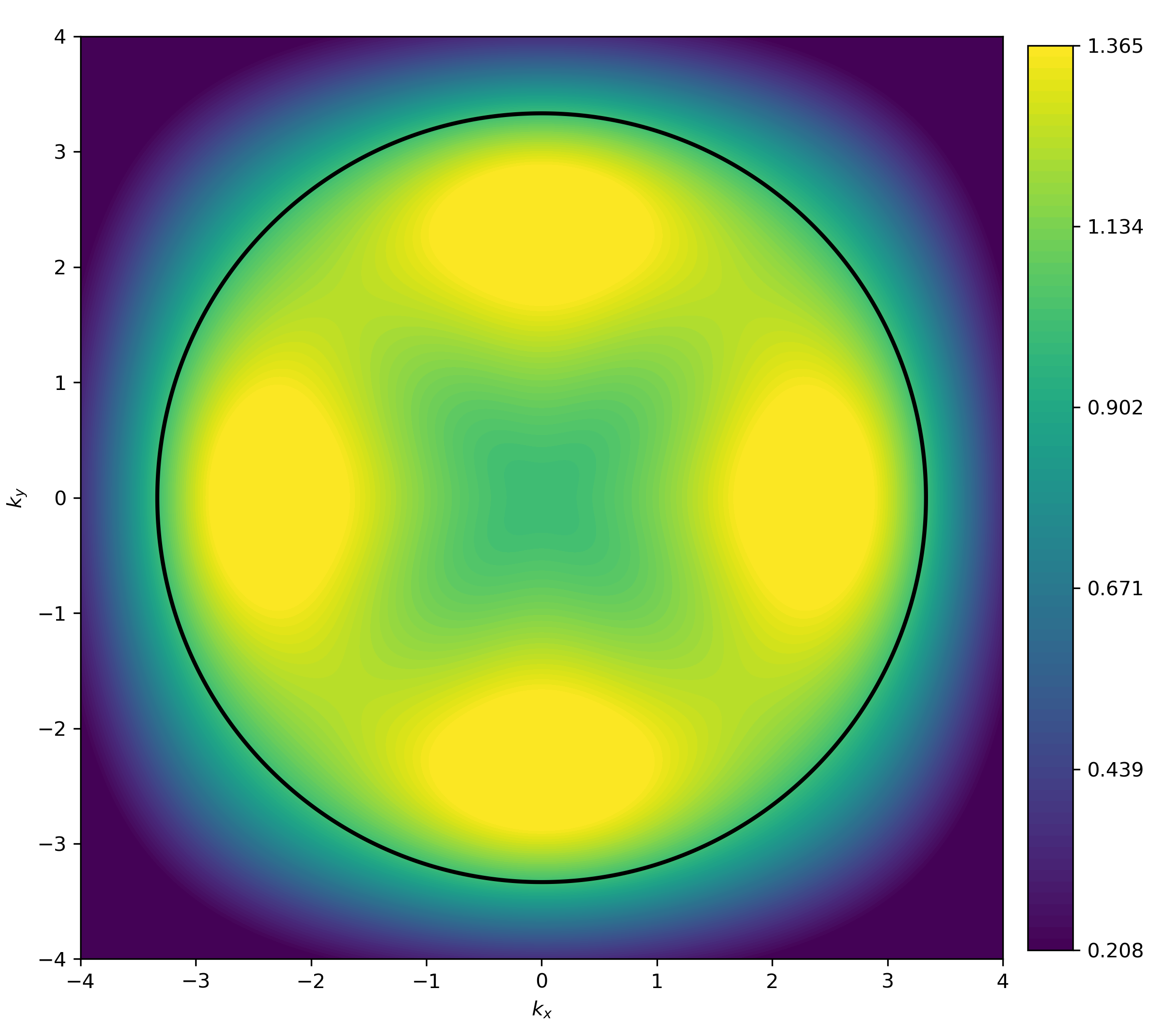}
			\end{subfigure}\hfill
			\begin{subfigure}{0.49\linewidth}
				\centering
				\includegraphics[width=\linewidth]{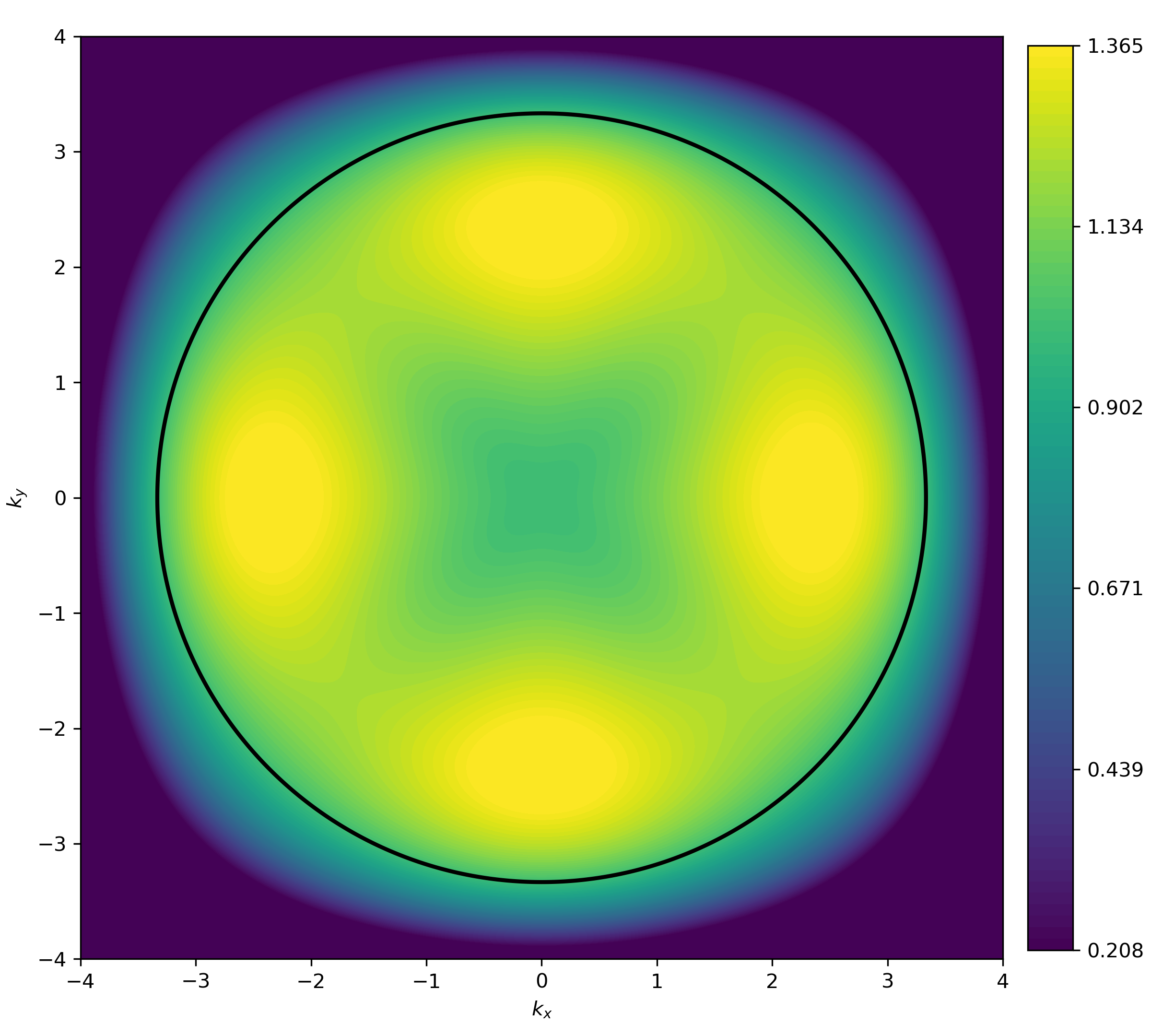}
			\end{subfigure}
			
			\vspace{0.6em}
			
			\begin{subfigure}{0.49\linewidth}
				\centering
				\includegraphics[width=\linewidth]{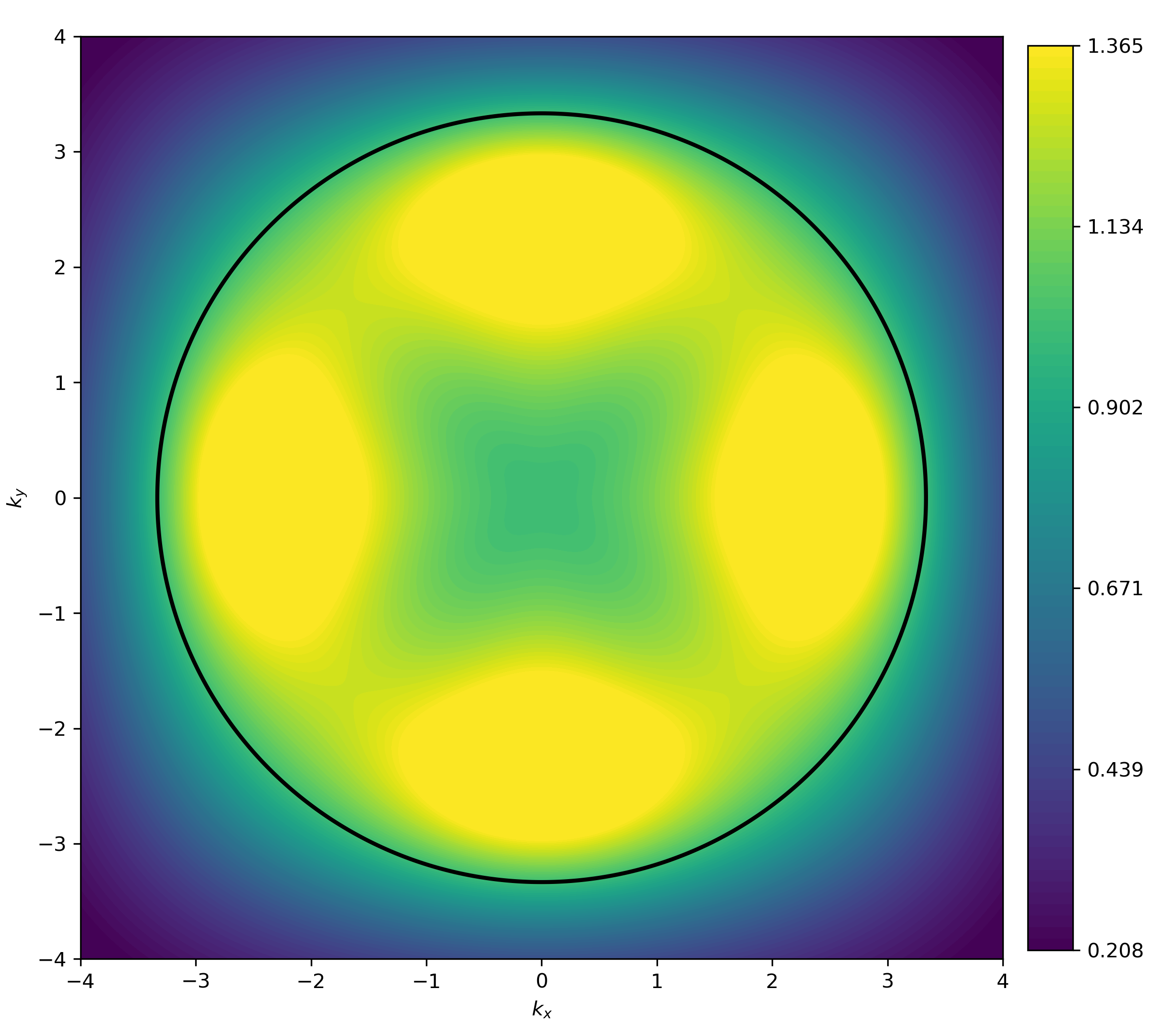}
			\end{subfigure}\hfill
			\begin{subfigure}{0.49\linewidth}
				\centering
				\includegraphics[width=\linewidth]{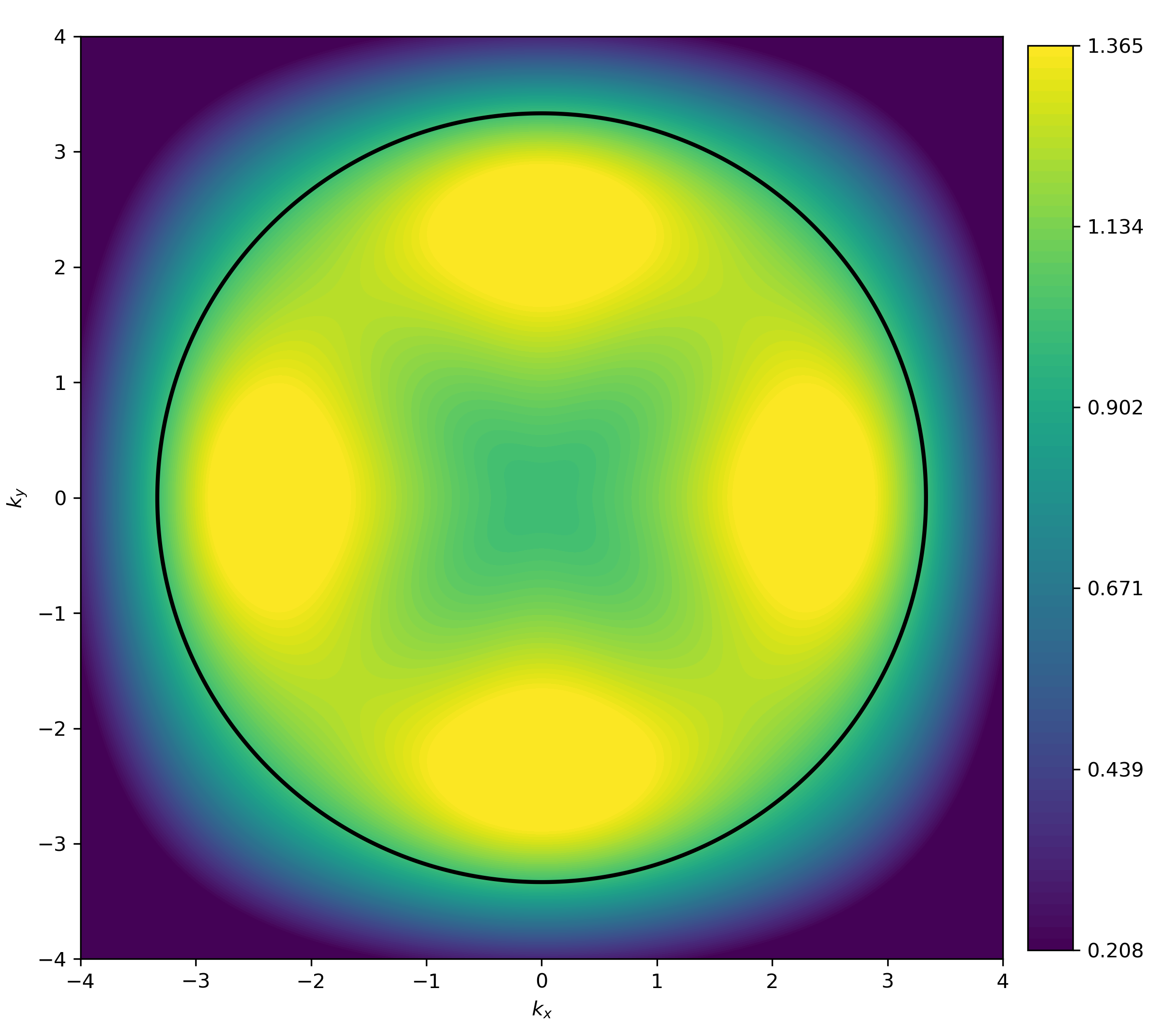}
			\end{subfigure}
			
		\end{minipage}%
	}
\caption{Discrete dispersion relation of the linearized anisotropic CH equation with $u_0=0$ and $\theta=0$. Top left: exact amplification factor; top right: explicit Euler; bottom left: implicit Euler; bottom right: midpoint method.}
\label{fig:comparison-aniso}
\end{figure}
\paragraph{Equilibrium shapes and missing orientations}
The Gibbs--Thomson relation describes the dependence of the chemical potential on the curvature of an interface. For a 2D interface,  it is described as
\[
\mu=\mu_0+v_m\bigl(\Gamma(\theta)+\Gamma_{\theta\theta}(\theta)\bigr)\,K,
\] where  $\Gamma(\theta)$ is the anisotropic surface energy and $\mu$ is the interfacial chemical potential.  At equilibrium, $\mu \equiv \mu_e$ is constant, and the equilibrium crystal corresponds to the Wulff shape of $\Gamma(\theta)$. A regular Wulff shape (with no missing orientations) requires the surface stiffness to be nonnegative, which implies
\begin{equation}\label{eq:stiffness}
\Gamma(\theta)+\Gamma_{\theta\theta}(\theta)\ge 0.
\end{equation}
If Eq.~\eqref{eq:stiffness} is violated over an angular interval, the associated high-energy orientations are missing from the equilibrium interface. Geometrically, the naive parametric Wulff curve then contains unstable branches, often called "ears" \cite{CahnHoffman1974,Herring1951,HoffmanCahn1972,Sekerka2005}. The physical Wulff shape is then obtained by convexifying the construction and eliminating these branches.

For the twofold anisotropy, \(\Gamma(\theta;\alpha)=1+\alpha\cos(2\theta)\), so that \(\Gamma+\Gamma_{\theta\theta}=1-3\alpha\cos(2\theta)\), and the corresponding critical value is \(\alpha_c=\frac13\). For the fourfold anisotropy, \(\Gamma(\theta;\alpha)=1+\alpha\cos(4\theta)\), so that \(\Gamma+\Gamma_{\theta\theta}=1-15\alpha\cos(4\theta)\), and the corresponding critical value is \(\alpha_c=\frac{1}{15}\), where \(\alpha_c\) denotes the critical value at which the stiffness~\eqref{eq:stiffness} first vanishes.
We next study how the anisotropy strength $\alpha$ affects the equilibrium morphology and the missing-orientation behavior. 
The results are illustrated in Fig.~\ref{fig:Alpha2-mix}-- \ref{fig:Alpha4-mix}. Different values of \(\alpha\) give rise to distinct interface shapes and may induce missing-orientation behavior.

\begin{figure}[htbp]
\centering
\begin{subfigure}[t]{0.32\textwidth}
	\centering
	\includegraphics[width=\textwidth]{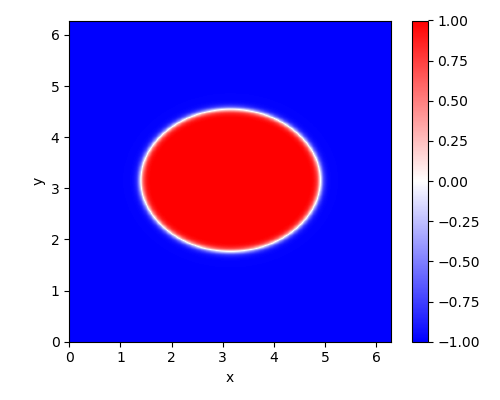}
\end{subfigure}\hfill
\begin{subfigure}[t]{0.66\textwidth}
	\centering
	\includegraphics[width=\textwidth]{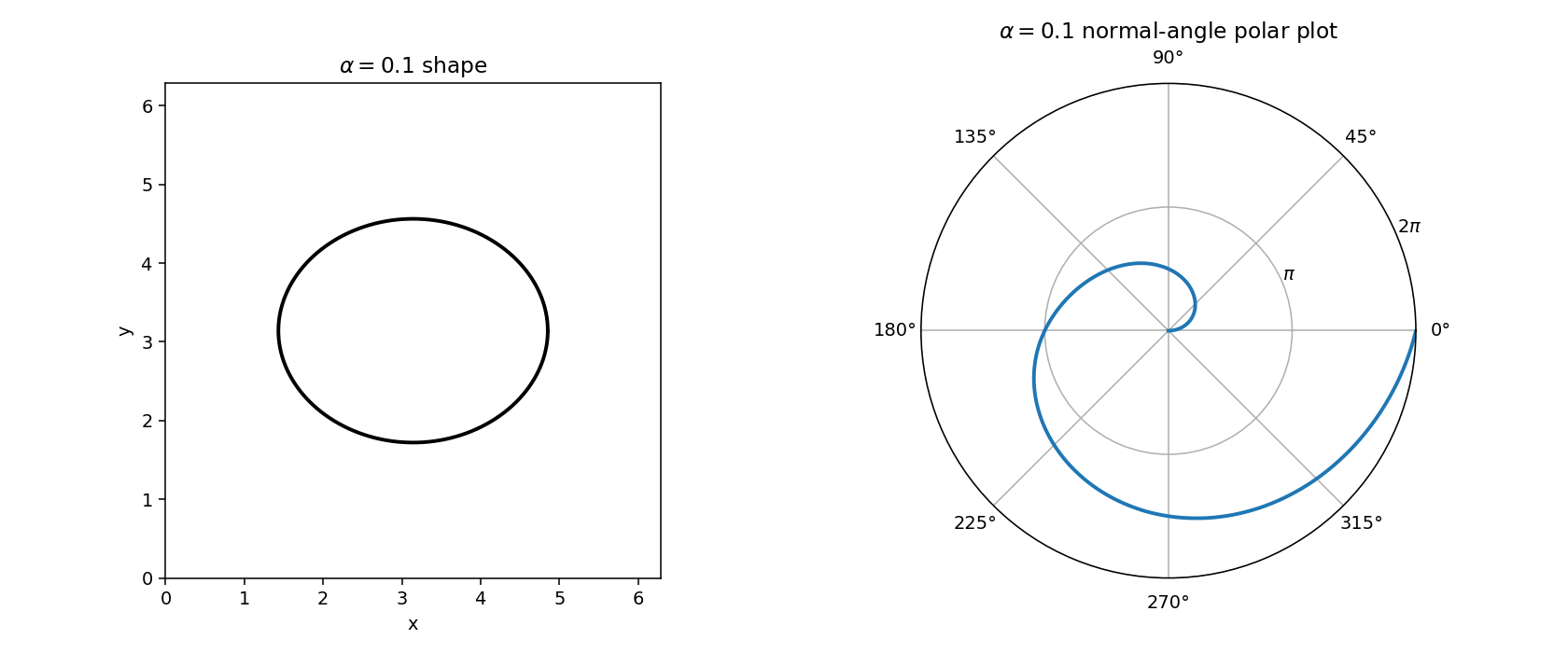}
\end{subfigure}

\vspace{0.5em}

\begin{subfigure}[t]{0.32\textwidth}
	\centering
	\includegraphics[width=\textwidth]{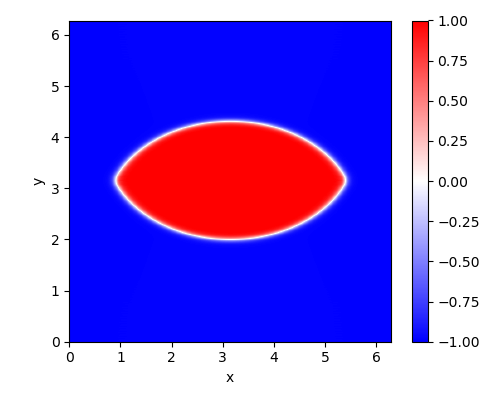}
\end{subfigure}\hfill
\begin{subfigure}[t]{0.66\textwidth}
	\centering
	\includegraphics[width=\textwidth]{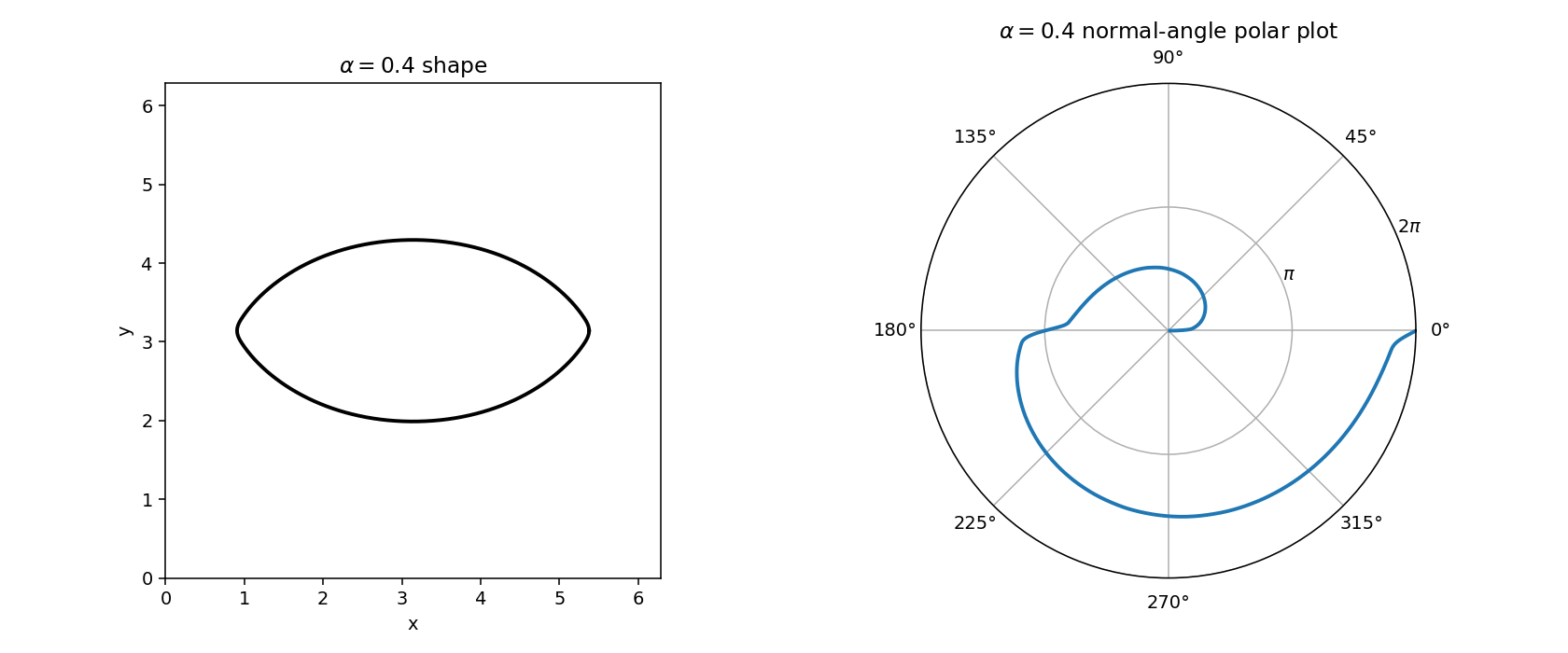}
\end{subfigure}
\caption{Morphologies and normal-angle polar plots for the twofold anisotropic CH model.}
\label{fig:Alpha2-mix}
\end{figure}

\begin{figure}[htbp]
\centering
\begin{subfigure}[t]{0.32\textwidth}
	\centering
	\includegraphics[width=\textwidth]{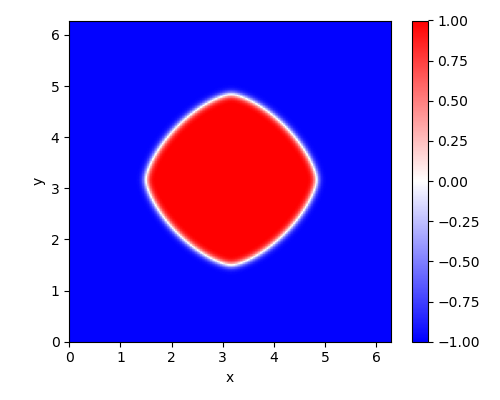}
\end{subfigure}\hfill
\begin{subfigure}[t]{0.66\textwidth}
	\centering
	\includegraphics[width=\textwidth]{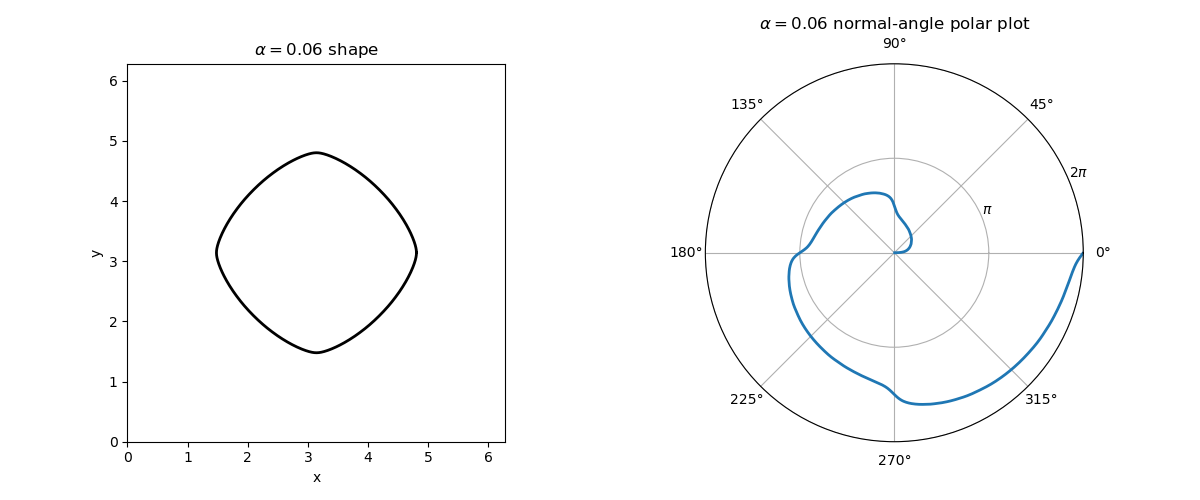}
\end{subfigure}

\vspace{0.5em}

\begin{subfigure}[t]{0.32\textwidth}
	\centering
	\includegraphics[width=\textwidth]{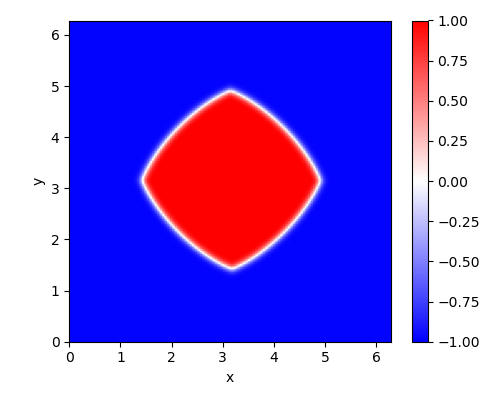}
\end{subfigure}\hfill
\begin{subfigure}[t]{0.66\textwidth}
	\centering
	\includegraphics[width=\textwidth]{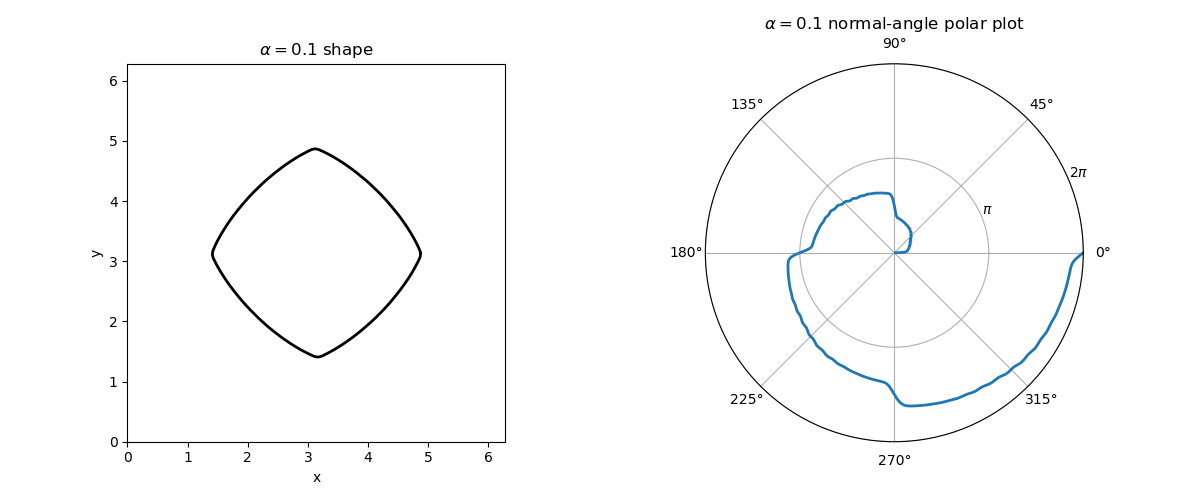}
\end{subfigure}
\caption{Morphologies and normal-angle polar plots for the fourfold anisotropic CH model.}
\label{fig:Alpha4-mix}
\end{figure}

Figures~\ref{fig:Alpha2-mix}--\ref{fig:Alpha4-mix} show the interface shapes and the corresponding normal-angle polar plots for different anisotropy strengths. It is observed that when Eq.~\eqref{eq:stiffness} is satisfied, the boundary is smooth and the normal-angle polar plot covers the angular range continuously, so that every orientation appears on the interface. In contrast, when Eq.~\eqref{eq:stiffness} is violated, the boundary becomes visibly faceted, and the normal-angle polar plot exhibits a gap-like loss of admissible directions. This indicates that some normal orientations are missing.
For the twofold and fourfold anisotropies, the changes occur across the
critical values \(\alpha_c=1/3\) and \(\alpha_c=1/15\), respectively. The gaps in	the normal-angle polar plots are therefore consistent with the stiffness criterion~\eqref{eq:stiffness}.
\section{Numerical Experiments}
In the previous sections, we developed fully discrete AIQ schemes for the isotropic and anisotropic CH equations. In this section, we examine the numerical performance of the proposed method. In the computations below, the AIQ scheme is implemented with the implicit midpoint rule, which is a second-order symplectic Runge--Kutta method. For the isotropic model, we compare AIQ with two representative second-order auxiliary-variable methods: the stabilized invariant energy quadratization (S-IEQ) method~\cite{Yang2017IEQ} and the scalar auxiliary variable (SAV) method~\cite{Shen2018SAV}. The comparison focuses on the temporal convergence of the solution error, the dissipation behavior of the discrete energy, and the convergence behavior of the original-energy error.

Unless otherwise specified, all numerical experiments are carried out on the periodic domain $\Omega=[0,2\pi]^2$. For the isotropic case, the initial condition is chosen as
\[
u_0(x,y)=0.05\left[\cos(3x)\cos(4y)
+
\bigl(\cos(4x)\cos(3y)\big)^2
+
\cos(x-5y)\cos(2x-y)\right].
\]
This smooth multi-mode initial condition is used as a benchmark for assessing the temporal accuracy, energy-dissipation behavior, and auxiliary-variable consistency of the numerical schemes. All schemes are implemented using the same spatial discretization and time-step sizes. 

To examine temporal accuracy, we compute numerical solutions with
$\Delta t=5\times10^{-5},\,2.5\times10^{-5},\,1.25\times10^{-5}$
on a fixed spatial grid. To further compare the energy behavior, we examine the energy evolution of the AIQ, S-IEQ, and SAV schemes. The results are shown in Fig.~\ref{fig:iso-compare}.
\begin{figure}[htbp]
\centering
\begin{subfigure}[t]{0.32\textwidth}
	\centering
	\includegraphics[width=\textwidth]{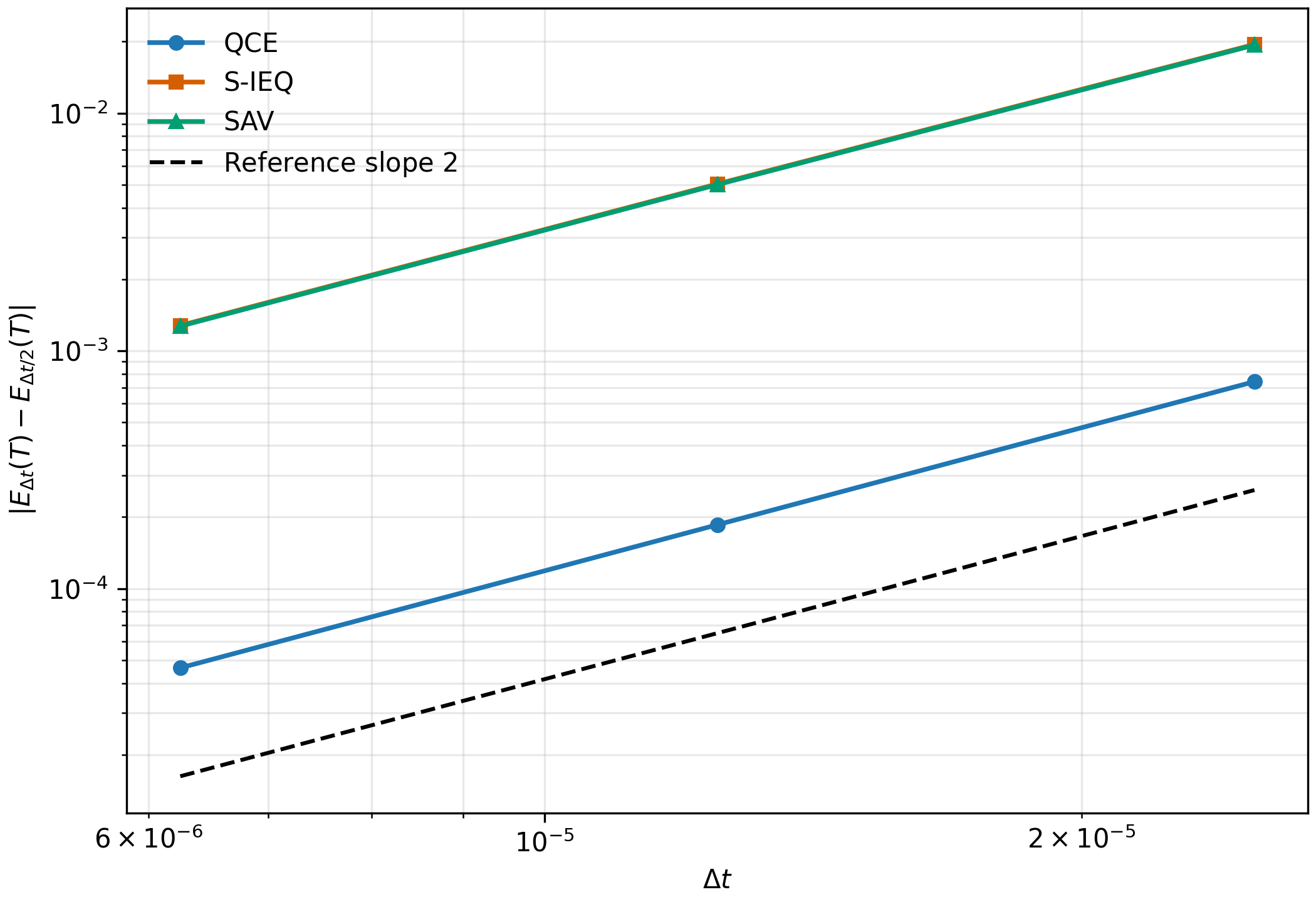}
	\caption{Solution-error}
\end{subfigure}
\hfill
\begin{subfigure}[t]{0.32\textwidth}
	\centering
	\makebox[\textwidth][c]{%
	\begin{tikzpicture}
		\node[inner sep=0pt] (iso-energy) {\includegraphics[width=\textwidth]{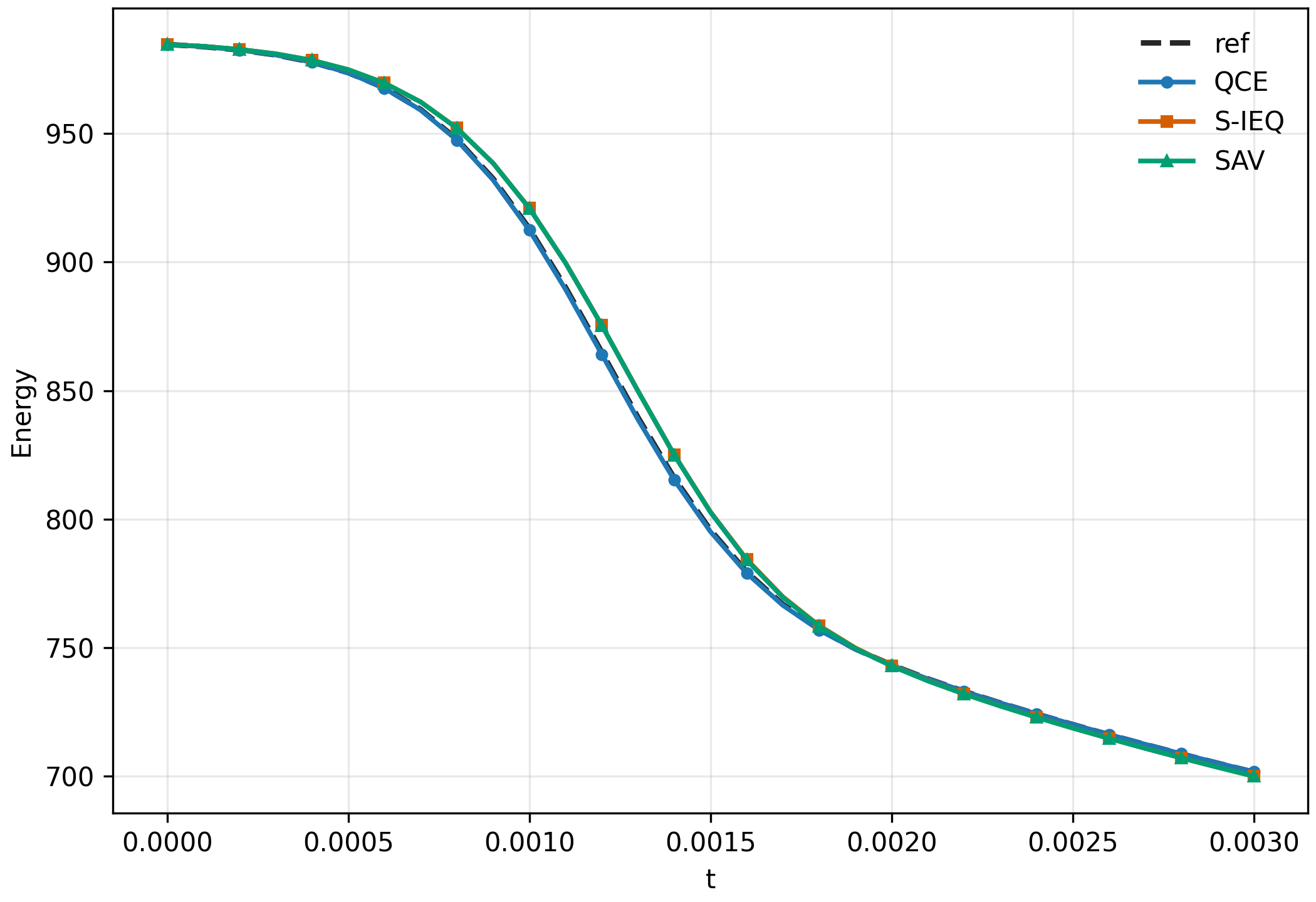}};
		\path[use as bounding box] (iso-energy.south west) rectangle (iso-energy.north east);
		\node[anchor=center,inner sep=0pt,draw=black,line width=0.2pt,fill=white] (iso-energy-zoom)
			at ([xshift=-0.16\textwidth,yshift=-0.10\textwidth]iso-energy.center)
			{\includegraphics[width=0.38\textwidth]{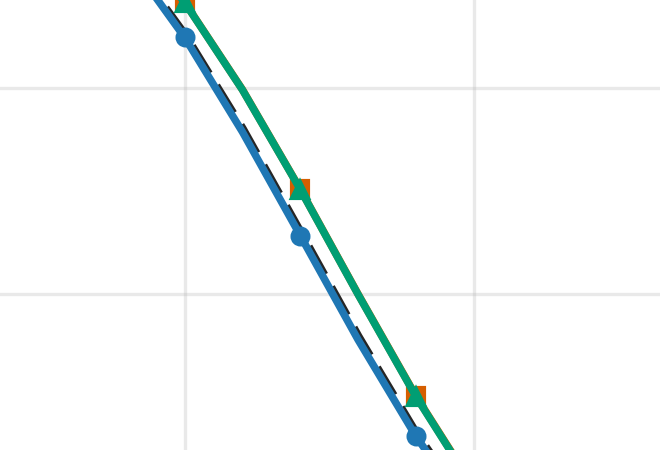}};
		\draw[-{Stealth[length=1.8mm]},line width=0.25pt]
			(iso-energy-zoom.east) -- ([xshift=0.00\textwidth,yshift=0.00\textwidth]iso-energy.center);
	\end{tikzpicture}
	}
	\caption{Original energy}
\end{subfigure}
\hfill
\begin{subfigure}[t]{0.32\textwidth}
	\centering
	\includegraphics[width=\textwidth]{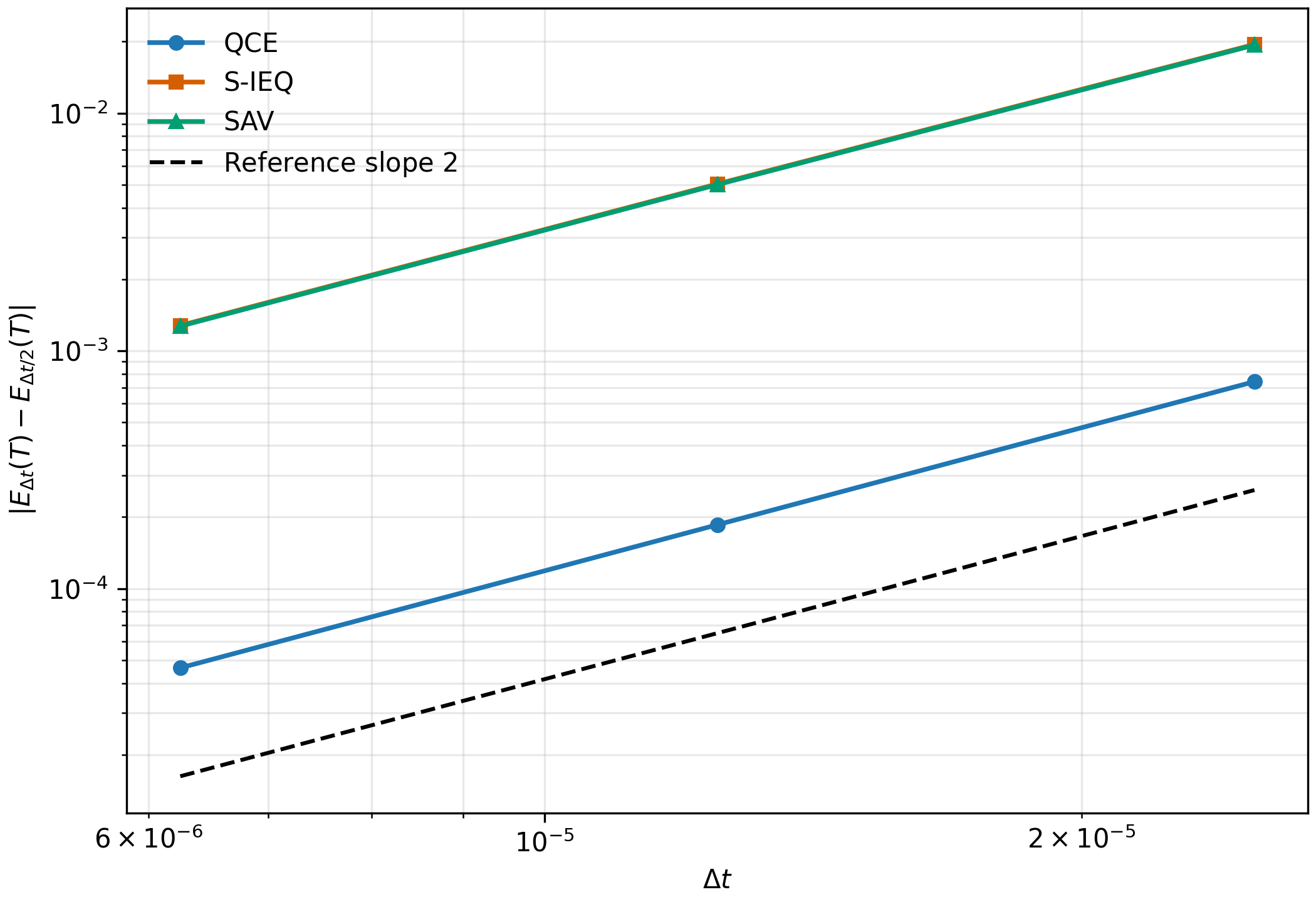}
	\caption{Original-energy error}
\end{subfigure}
\caption{Numerical comparison for the isotropic CH equation.}
\label{fig:iso-compare}
\end{figure}

Figure~\ref{fig:iso-compare}(a) shows that all three methods converge at approximately second order in time, while AIQ has a smaller error constant. Compared with the reference solution, Figure~\ref{fig:iso-compare}(b) indicates that the AIQ method remains closer to the reference energy evolution than those of S-IEQ and SAV. All methods preserve the decreasing trend of the energy, but they differ in their accuracy for the original energy evolution. In particular, the AIQ scheme tracks the reference energy curve more closely. Figure~\ref{fig:iso-compare}(c) shows the temporal convergence of the original-energy error. These results indicate that AIQ better preserves the original energy structure in this test.

We also verify that the AIQ formulation can be combined with higher-order SRK time discretizations. As an example, we consider the fourth-order AIQ scheme obtained from the two-stage Gauss--Legendre tableau The test uses the same smooth isotropic initial condition with \(T=2\times10^{-3}\), and \(\Delta t=2.5\times10^{-5},1.25\times10^{-5},6.25\times10^{-6},3.125\times10^{-6}\). Figure~\ref{fig:QCE-gauss4-order} shows the expected fourth-order temporal convergence.
\begin{figure}[htbp]
	\centering
	\includegraphics[width=0.62\textwidth]{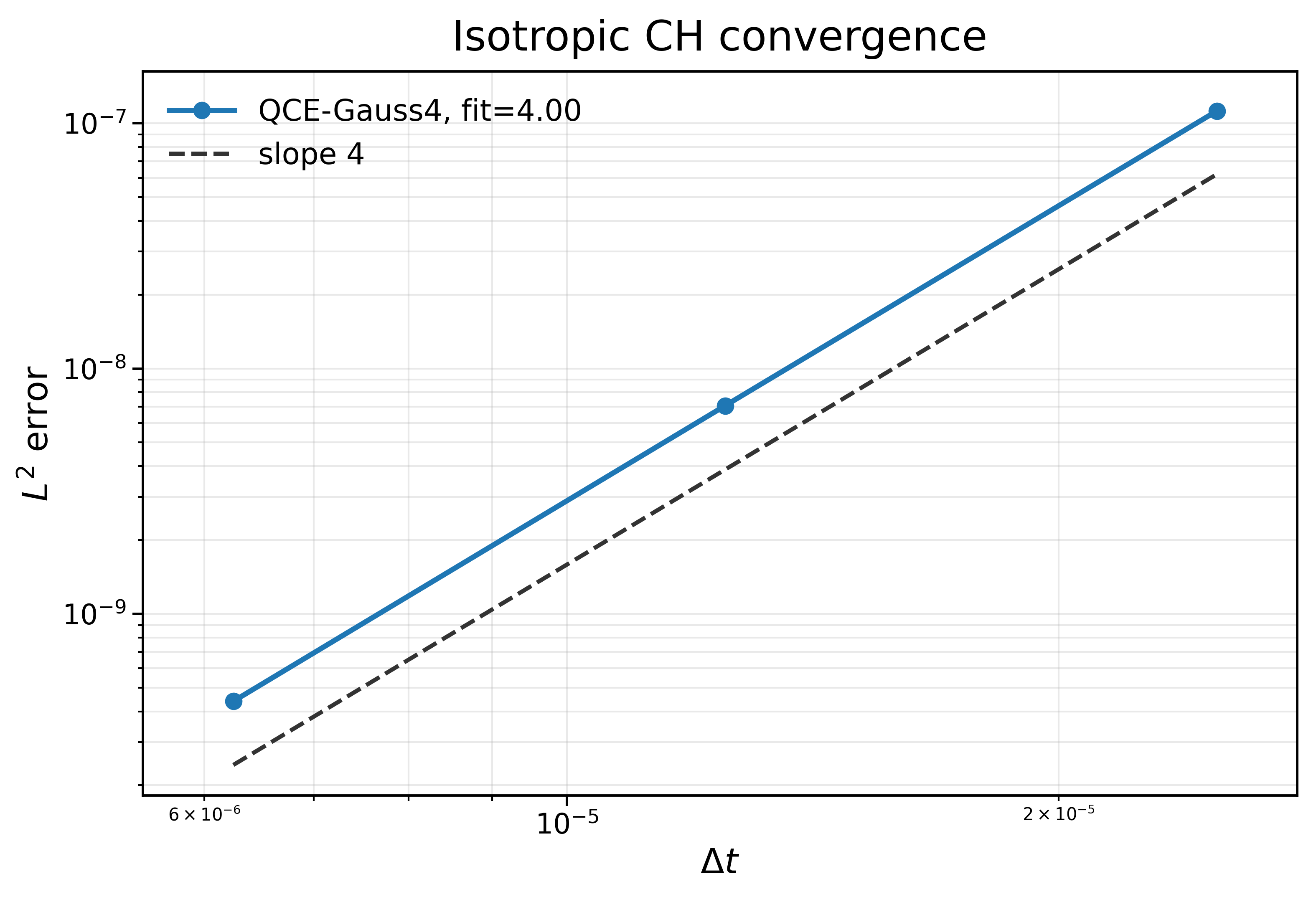}
	\caption{Self-convergence test for the fourth-order AIQ scheme based on the two-stage Gauss--Legendre SRK method.}
	\label{fig:QCE-gauss4-order}
\end{figure}

We next consider the anisotropic CH model~\eqref{eq:CH-ani} and examine the effect of anisotropy on the interfacial dynamics. Unless otherwise specified, we take $\alpha=0.1$ and $\beta=6\times10^{-4}$ in the following experiments.
\paragraph{Single droplet}
We initialize $u$ as a single droplet centered at $(\pi,\pi)$ which is given by
\[
u_0(x,y)=\tanh\!\Bigl(\frac{\sqrt{(x-\pi)^2+(y-\pi)^2}-0.5\pi}{1.2\varepsilon}\Bigr).
\]
We use a time step size $\Delta t=10^{-4}$. Figure~\ref{fig:I-single} shows the phase-field evolution and the corresponding discrete free energy, respectively.

\begin{figure}[htbp]
\centering
\begin{subfigure}[t]{0.55\textwidth}
	\centering
	\includegraphics[width=\textwidth]{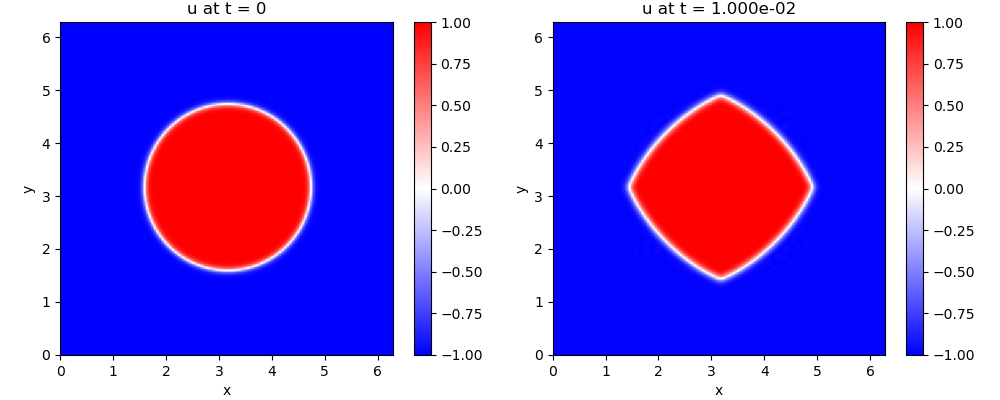}
	\caption{Evolution of the phase-field}
\end{subfigure}
\hfill
\begin{subfigure}[t]{0.40\textwidth}
	\centering
	\includegraphics[width=\textwidth]{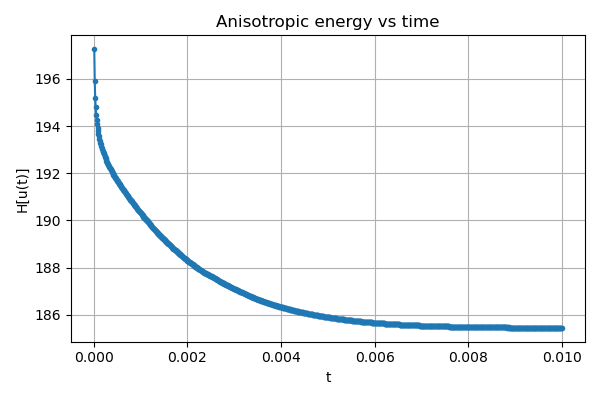}
	\caption{Evolution of discrete energy}
\end{subfigure}
\caption{Single-droplet test: phase-field evolution and discrete energy decay.}
\label{fig:I-single}
\end{figure}

Figure~\ref{fig:I-single} illustrates the evolution of a single droplet. The initially circular interface quickly loses rotational symmetry and develops flat edges and sharp corners. This behavior is expected because the anisotropy function $\Gamma(\theta)$ assigns different interfacial energies to different orientations, so the interface motion favors the directions that minimize the anisotropic surface energy.

Meanwhile, Fig.~\ref{fig:I-single}(b) shows that the discrete free energy decreases monotonically in time. This monotone decay indicates that the proposed scheme preserves the dissipative structure of the CH dynamics.

Under the same initial condition, we compare the proposed AIQ scheme with the S-IEQ scheme in terms of the solution error, original energy evolution, and original-energy error.

\begin{figure}[htbp]
\centering
\begin{subfigure}[t]{0.32\textwidth}
	\centering
	\includegraphics[width=\textwidth]{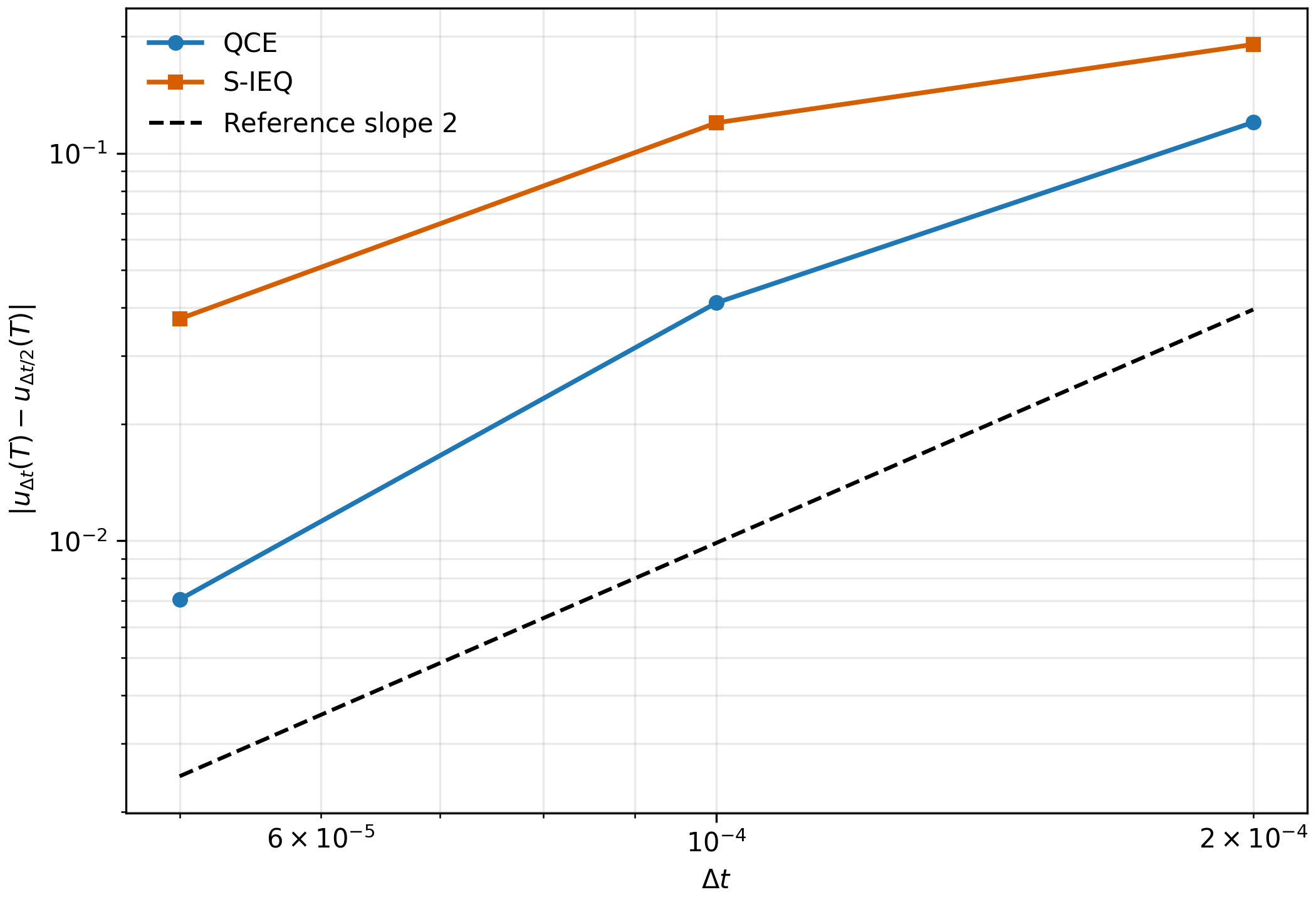}
	\caption{Solution-error}
\end{subfigure}
\hfill
\begin{subfigure}[t]{0.32\textwidth}
	\centering
	\makebox[\textwidth][c]{%
	\begin{tikzpicture}
		\node[inner sep=0pt] (ani-energy) {\includegraphics[width=1.12\textwidth]{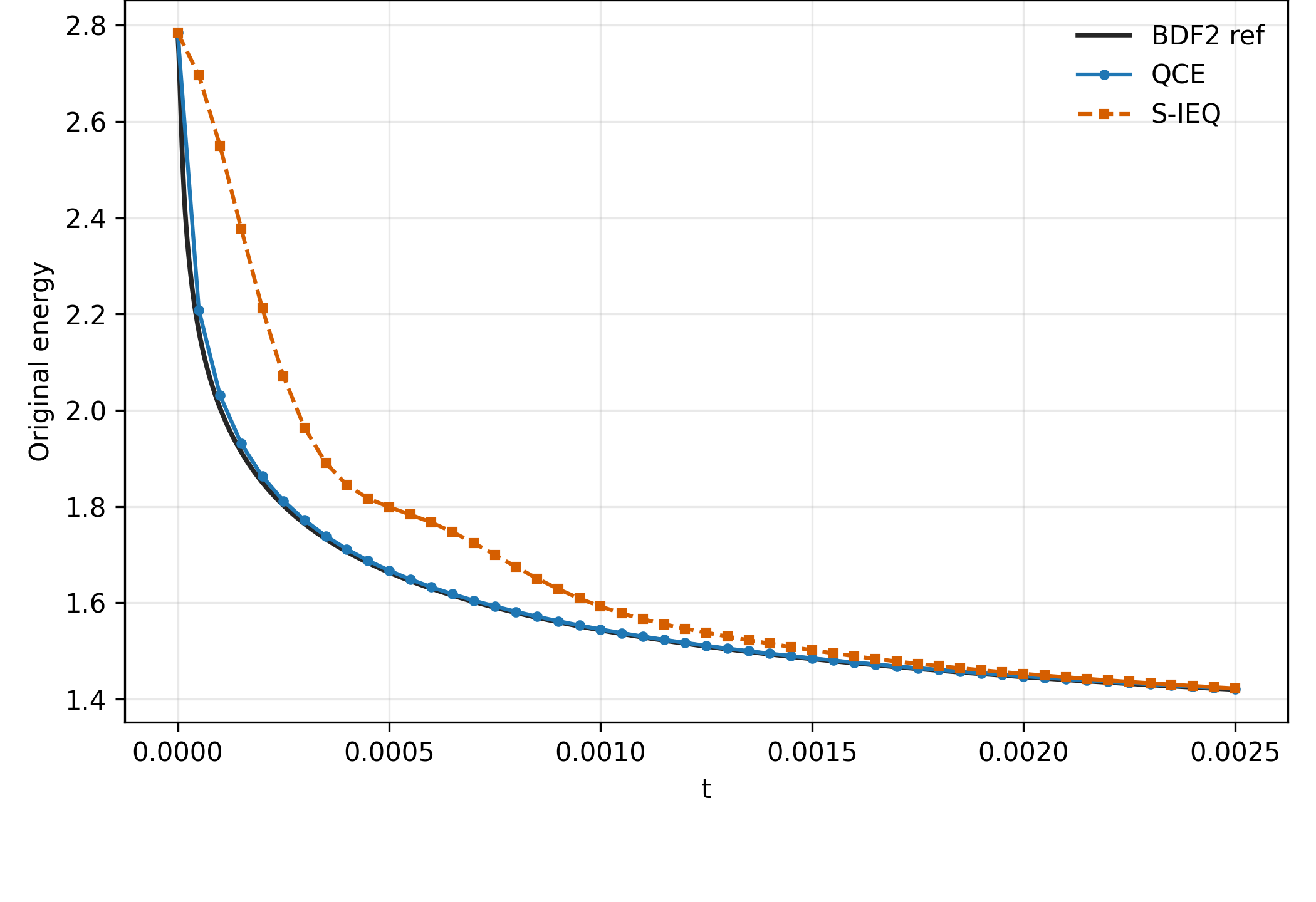}};
		\path[use as bounding box] (ani-energy.south west) rectangle (ani-energy.north east);
		\node[anchor=center,inner sep=0pt,draw=black,line width=0.2pt,fill=white] (ani-energy-zoom)
			at ([xshift=0.23\textwidth,yshift=0.02\textwidth]ani-energy.center)
			{\includegraphics[width=0.38\textwidth]{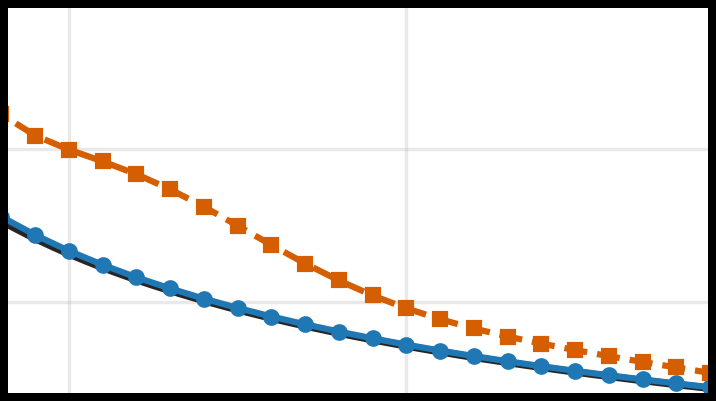}};
		\draw[-{Stealth[length=1.8mm]},line width=0.25pt]
			(ani-energy-zoom.south west) -- ([xshift=-0.04\textwidth,yshift=-0.15\textwidth]ani-energy.center);
	\end{tikzpicture}
	}
	\caption{Original energy}
\end{subfigure}
\hfill
\begin{subfigure}[t]{0.32\textwidth}
	\centering
	\includegraphics[width=\textwidth]{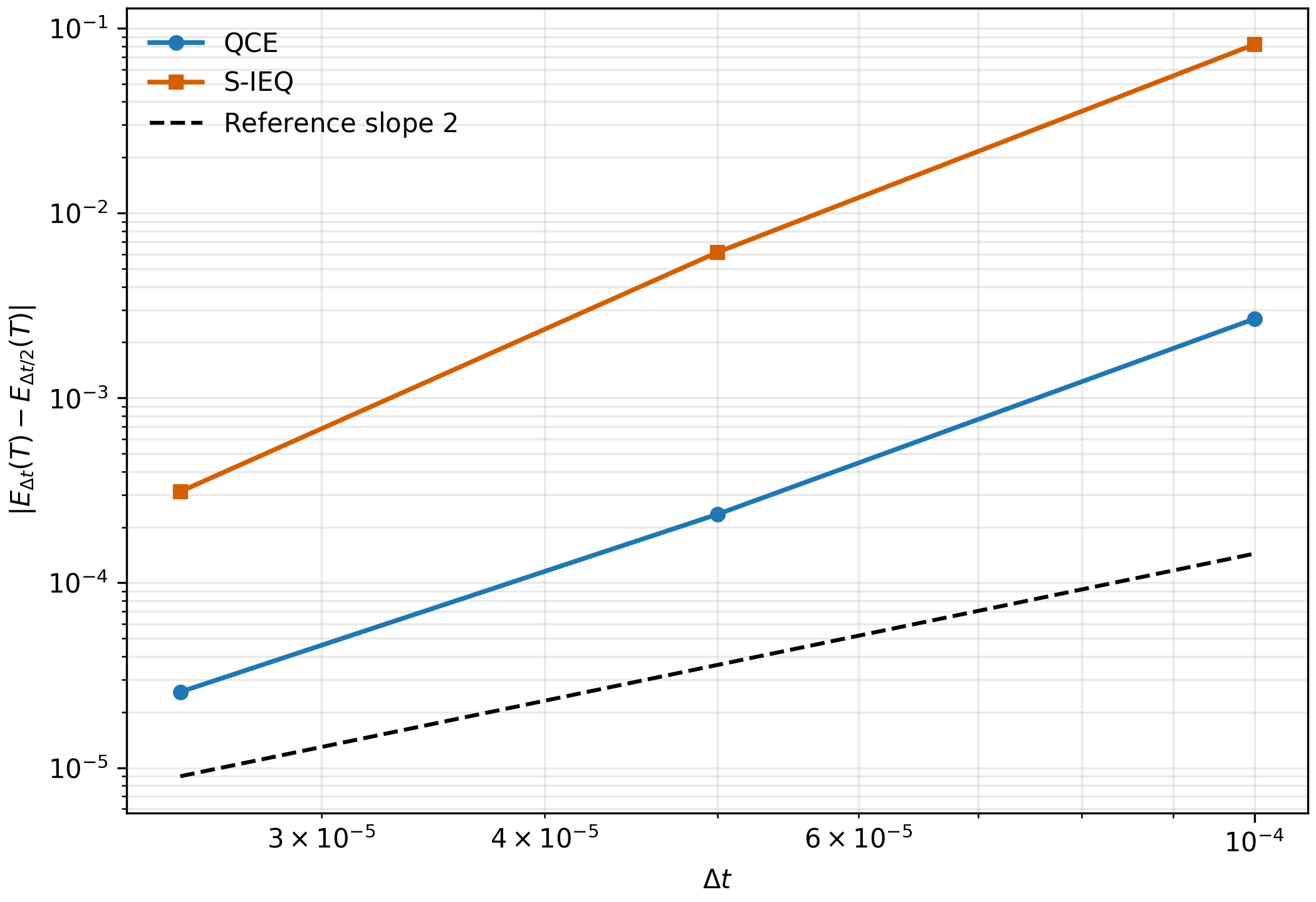}
	\caption{Original-energy error}
\end{subfigure}
\caption{Numerical comparison for the anisotropic CH equation.}
\label{fig:ani-compare}
\end{figure}

Figure~\ref{fig:ani-compare} compares the AIQ and S-IEQ schemes for the anisotropic CH equation in terms of solution-error convergence, original energy dissipation, and original-energy error convergence. Figure~\ref{fig:ani-compare}(a) shows that both schemes follow the reference second-order slope, while AIQ yields smaller solution errors. Figure~\ref{fig:ani-compare}(b) shows that the original discrete energy decreases in time, and the AIQ energy curve stays closer to the reference energy evolution. Figure~\ref{fig:ani-compare}(c) further indicates that AIQ gives smaller original-energy errors than S-IEQ. These results suggest that the AIQ formulation provides a more accurate approximation of the original anisotropic energy structure.
\paragraph{Double droplets}
We consider the following two-droplet initial condition:
\begin{equation}\label{eg:two-droplet}
u_0(x,y)=\max\{u_1(x,y),u_2(x,y)\},
\end{equation}
where \(u_i(x,y)=\tanh\!\left(
\frac{\sqrt{(x-x_i)^2+(y-y_i)^2}-R_i}{1.2\,\varepsilon}
\right)\),
with
\[
(x_1,y_1,R_1)=(0.8\pi,\,1.02\pi,\,0.5\pi),
\qquad
(x_2,y_2,R_2)=(1.57\pi,\,0.98\pi,\,0.2\pi).
\]
The resulting coalescence dynamics and the corresponding energy evolution are
shown in Fig.~\ref{fig:I-double}.

%
\begin{figure}[htbp] \centering \begin{subfigure}[t]{0.62\textwidth} \centering \includegraphics[width=\textwidth]{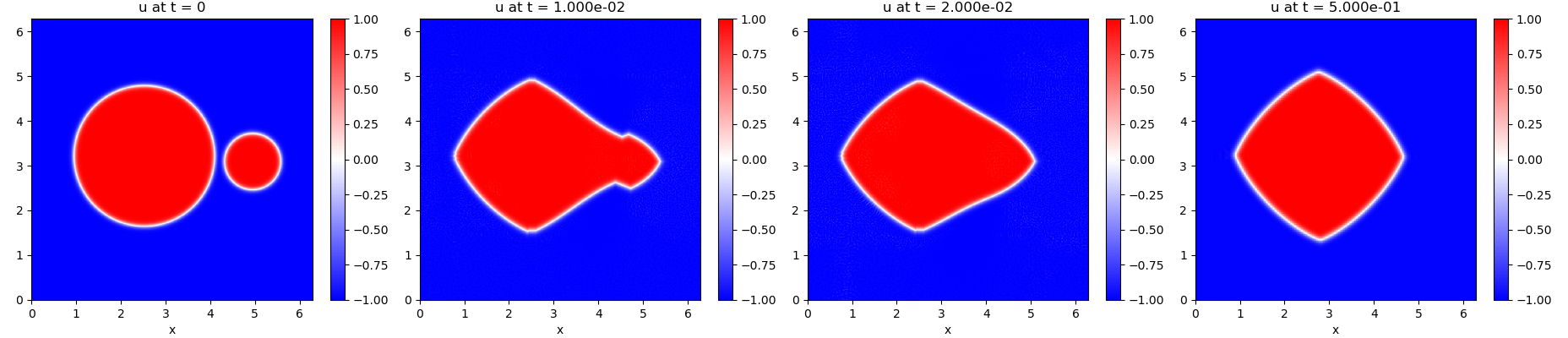} \caption{Evolution of the phase-field} 
\end{subfigure}
\begin{subfigure}[t]{0.34\textwidth} \centering \includegraphics[width=\textwidth]{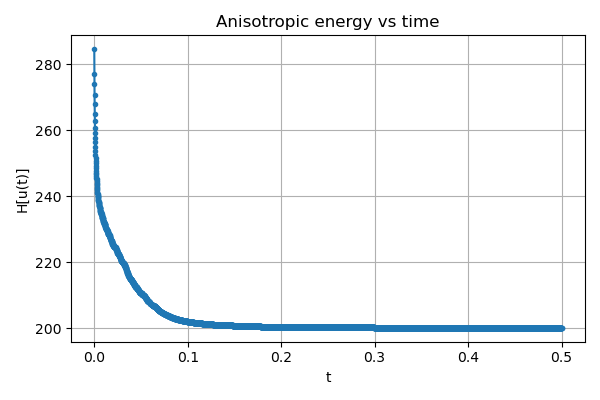} \caption{Evolution of discrete energy.}
\end{subfigure} 
\caption{Double-droplet test: phase-field evolution and discrete energy decay.} \label{fig:I-double} 
\end{figure}
Starting from two circular droplets, the solution rapidly develops straight facets and sharp corners aligned with the fourfold preferred directions prescribed by $\Gamma(\theta)$, and the merged droplet evolves toward a diamond-like Wulff shape. Meanwhile, the total mass remains conserved and the discrete free energy decreases monotonically throughout the simulation. This is consistent with the mass-conservation and energy-dissipation properties of the fully discrete AIQ scheme.

\paragraph{Random initial value}
We next consider a randomly perturbed initial field to test anisotropic pattern selection from nonstructured data. We take

\[
u_0(x,y)=\xi(x,y),
\qquad \xi(x,y)\sim \mathcal U(-1,1).
\]
We need to take a smaller time step $\Delta t=10^{-6}$ due to the stronger nonlinearity induced by the orientation-dependent factor, and the results are shown in
Fig.~\ref{fig:I-rand-phase}.

\begin{figure}[htbp]
\centering
\includegraphics[width=0.8\textwidth]{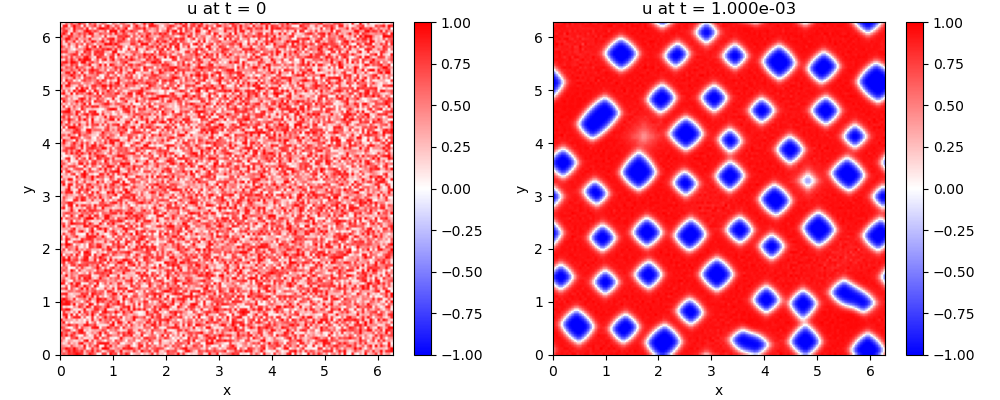}
\caption{Phase-field evolution for the random test. Left: initial condition at $t=0$. Right: solution at $t=0.001$.}
\label{fig:I-rand-phase}
\end{figure}

Fig.~\ref{fig:I-rand-phase} shows that the random perturbations quickly evolve into a fourfold-symmetric pattern, and the interfaces develop diamond-like facets aligned with the preferred orientations prescribed by \(\Gamma(\theta)\).


\section{Concluding remarks}
In this paper, we proposed a quadratic reformulation framework for rational-like energy functions. Based on this framework, we developed the Quadratic Conserving Elevation (AIQ) method by introducing suitable auxiliary variables and applying the implicit midpoint rule to the corresponding extended system. We have applied this method to Cahn--Hilliard (CH) equations with rational-like free-energy terms and proved that the resulting schemes preserve the original energy dissipation law.

We also analyzed the discrete dispersion relation of the proposed schemes and investigated their consistency with the continuous dynamics. Numerical experiments confirmed the expected spinodal decomposition, coarsening behavior, and second-order temporal accuracy. For anisotropic CH models, the method was further shown to capture missing orientations associated with different anisotropic energy functions. Simulations with various initial conditions illustrated phase separation, long-time coarsening dynamics, and anisotropic evolution.

Several directions remain for future work. It would be interesting to extend the original-energy AIQ strategy to logarithmic Flory--Huggins free energies, more general anisotropic surface energies, and anisotropic mobility operators. Another important direction is to improve the computational efficiency of the extended systems. Possible approaches include designing preconditioned iterative solvers for the nonlinear algebraic systems, reducing redundant auxiliary variables, and developing adaptive time-stepping strategies. These techniques may further accelerate the AIQ schemes while retaining their original-energy consistency.

\bibliographystyle{plain}      

\bibliography{references}

\end{document}